\RequirePackage{fixltx2e}
\documentclass[10pt]{article}  
\usepackage[utf8]{inputenc}
\usepackage[T1]{fontenc}
\usepackage{amsmath}
\DeclareFontFamily{U}{mathx}{}
\DeclareFontShape{U}{mathx}{m}{n}{<-> mathx10}{}
\DeclareSymbolFont{mathx}{U}{mathx}{m}{n}
\DeclareMathAccent{\widehat}{0}{mathx}{"70}
\DeclareMathAccent{\widecheck}{0}{mathx}{"71}

\usepackage{float}
\usepackage{amsmath,amssymb,amsfonts,graphicx,bm}
\usepackage{booktabs}
\usepackage{hyperref}
\hypersetup{colorlinks=false,pdfborder={0 0 0}}
\usepackage{geometry}
\usepackage{amsthm}
\usepackage{color}
\usepackage[normalem]{ulem}
\newtheorem{claim}{Claim}[section]
\newtheorem{proposition}[claim]{Proposition}
\newtheorem{corollary}[claim]{Corollary}
\theoremstyle{remark}
\newtheorem{remark}{Remark}[section]

\numberwithin{equation}{section}
\numberwithin{figure}{section}
\numberwithin{table}{section}

\newcommand \bei {\begin{itemize}}
\newcommand \eei {\end{itemize}}

\newcommand \bel{\begin{equation}\label}
\newcommand \del \partial
\newcommand \be {\begin{equation}}
\newcommand \ee {\end{equation}}
\newcommand \bse {\begin{subequations}}
\newcommand \ese {\end{subequations}}

\newcommand \eps \epsilon
 
\definecolor{darkblue}{RGB}{0,0,139}

\begin{document}

\title{Physics-informed token transformer methodology for nonlinear balance laws. I. 
Schwarzschild--Burgers fluid flows}

\author{\begin{tabular}{c}
\large Philippe G. LeFloch
and Shuyang Xiang
\\[0.5em]
Laboratoire Jacques-Louis Lions
\\
and Centre National de la Recherche Scientifique,
\\
{{Sorbonne Universit\'e, 4~Place Jussieu, 75252 Paris, France.}}\\
{{Emails: {\sl contact@philippelefloch.org, vanillaxiangshuyang@gmail.com}.}}
\end{tabular}}

\maketitle


\begin{abstract}
We introduce a Physics-Informed Token Transformer (PITT) methodology for nonlinear
hyperbolic balance laws in one space dimension, using piecewise
steady-state profiles for the representation of approximate weak solutions. The method combines
symbolic equation tokenization, a Fourier neural operator encoder, an
explicit Rankine--Hugoniot law for shock motion, and a learned correction
term. For clarity, we present it here for the relativistic Schwarzschild--Burgers equation, a
scalar model for spherically symmetric fluid flows on a Schwarzschild
background. For this model the steady-state invariant and the generalized
Riemann solutions are explicit, and they can therefore be built into the
neural evolution. In particular, the leading discontinuities are advanced by
the analytical jump condition, while the learned part reconstructs smooth
regions, rarefaction fans, geometric dependence, and finite-resolution
effects. The method is designed to locate wave fronts accurately and to preserve the relevant steady states. We test our PITT method on moving shocks, stationary shocks, rarefaction
waves, and compare it with a standard high-order finite-volume approximation. 
We also analyze the standard Burgers limit (when the Schwarzschild mass tends to zero). 
The Rankine--Hugoniot prior plays the dominant role in these tests, while
equation tokenization gives a systematic additional gain. The method is relevant for problems involving geometric effects and/or complex shock-wave dynamics, and is used here to study the long-time dynamics of perturbations of steady-state solutions. In particular, we exhibit an asymptotic law of propagation for the shock location of perturbed steady-state flows. 
\end{abstract} 


{\small
 
\setcounter{secnumdepth}{2} 
\setcounter{tocdepth}{2}
\tableofcontents

} 


\section{Introduction}
\label{section=1}

\subsection{Main purpose}
\label{section=1--1}

\subsubsection{Discontinuities and admissibility}

Learning discontinuous solutions of nonlinear hyperbolic conservation laws
and balance laws remains a major challenge for machine learning techniques, especially when one seeks accurate wave-front locations and physically admissible weak solutions with complex structure. 
Neural operators and PINN-type models\footnote{The reader not familiar with the machine learning terminology can refer to \autoref{section=A}.} are effective for many smooth
evolution problems, but shock waves, contact discontinuities, and
rarefaction fans require additional structure from the theory of weak
solutions~\cite{Lax1957,Lax1973,Dafermos2010,LeFloch2002}, including Rankine--Hugoniot jump relations, entropy admissibility criteria, and Riemann problems.
Without imposing such structure, smooth neural approximations tend to smear jumps,
misplace shock fronts, or violate conservation near discontinuities; in balance laws, they may also confuse genuine dynamics with source-driven or geometric effects. Fourier neural operators~\cite{Li2021} and
DeepONets~\cite{Lu2021} are designed to learn mappings between function
spaces{, within the broader neural-operator viewpoint~\cite{Kovachki2023}},
while transformer-based architectures introduce \emph{global} interactions between
spatial tokens~\cite{Vaswani2017,Cao2021}. These representations are useful, but they do not by themselves encode the admissibility mechanisms that determine discontinuous solutions of hyperbolic problems.


\subsubsection{Equation tokenization as physics conditioning}

The Physics-Informed Token Transformer approach introduced by Lorsung et
al.~\cite{Lorsung2024}, hereafter abbreviated as PITT, provides a
natural way to inject the structure of a partial differential equation (PDE) into
a neural update operator, not only through a \emph{residual loss} but also
through the \emph{internal representation} used by the model. Its main idea is to
represent differential operators, unknowns, parameters, and boundary or
initial data through \emph{symbolic equation tokens} processed by a
Transformer. The resulting physics representation imposes restrictions on the learned
time evolution and improves generalization across equations and parameter
regimes; namely, the network receives both sampled solution data and a
\emph{structured description} of the equation that produced them. In the present
work we retain this tokenization principle, but adapt it to a substantially
different regime: nonlinear hyperbolic balance laws whose physically
relevant solutions contain shocks and other discontinuities, so that
the symbolic representation must be coupled with admissibility and
wave-propagation mechanisms from the theory of weak solutions.


\subsubsection{A Rankine--Hugoniot prior for shock motion}

At the numerical level, the main point is to couple the tokenized equation
to an explicit Rankine--Hugoniot law and to \emph{equilibrium variables}
adapted to describe steady-state solutions to the underlying balance law. The shock speed is not learned from
data: it is computed from the analytical jump condition associated with the
flux structure, using the left and right traces selected
by the steady-state parametrization. The network is then used for the
remaining part of the evolution, namely the smooth or piecewise smooth
field reconstruction, the dependence on physical and geometric parameters,
and a correction for finite-resolution nonlinear and geometric effects.
Thus the motion of the discontinuity is separated from the neural field
approximation while the training data are used to resolve the
regular part of the solution.


\subsubsection{A class of fluid models with complex structure}

The approach is intended for nonlinear balance laws arising in classical
and relativistic fluid dynamics, especially when the difficulty comes from
curved geometry, non-trivial equilibria, systems with several wave
families, or interacting discontinuities. The scalar
\emph{Schwarzschild--Burgers equation} proposed by LeFloch et al.~\cite{LeFlochMakhlofOkutmustur2012,LeFlochMakhlof2014} 
is used here as a model problem for developing our PITT methodology: it admits explicit steady-state invariants and generalized Riemann solutions, while retaining source terms, stationary shocks, moving
shocks, and long-time relaxation. These exact structures are incorporated
directly into the representation, so that stationary solutions and
stationary discontinuities are treated as part of the model rather than as features to be captured numerically by an external scheme. The corresponding treatment of Schwarzschild--Euler
fluid flows is deferred to a follow-up paper~\cite{LeFlochXiang-2026b}.


\subsubsection{Relation with physics-informed neural methods}

This paper is related to the methodology of physics-informed neural networks (PINNs), such as~\cite{Raissi2019},
and their extensions to discontinuous solutions proposed in 
\cite{Mao2020,DeRyck2024,Liu2023,FerrerSanchez2024,Oubarka2026}.
Recent proposals concern incorporating Godunov or Riemann information,
conservative flux constraints, anti-diffusive evolution, and neural operators
adapted to hyperbolic problems~\cite{Cassia2025,Y-Liu-1,Kim2025,Liu2024,Liu2025,Lorin2024,
Patsatzis2024,Peyvan2024,Wang2025}. Time-continuous reconstruction of
fluid simulations is another related learning task~\cite{Pronk2025}.
These methods enrich the training loss with \emph{differential residuals} or
\emph{modified weights} near large gradients of the solutions, and have produced accurate results
for Riemann problems and hydrodynamic discontinuities. Our
strategy is different: the main hyperbolic constraint is imposed \emph{at the level of the
rollout mechanism itself.} In particular, the Rankine--Hugoniot condition
is not only penalized in a loss function but is \emph{used as an explicit
dynamical prior inside the architecture}. The resulting method combines
symbolic equation tokenization, neural operator encoding, equilibrium
variables, and analytical shock dynamics in a single construction for
discontinuous solutions of nonlinear hyperbolic balance laws,
thereby connecting operator learning with the classical structure of
front-tracking and generalized Riemann solvers.


\subsection{Main results}
\label{section=1--2}

\subsubsection{\texorpdfstring{A scalar geometric test model}{A scalar geometric test model}}

The proposed PITT method is designed to incorporate information from both the
governing equations and the underlying physics. The main objective
of the present article is therefore to develop the methodology in a scalar but genuinely
geometric setting, namely for Schwarzschild--Burgers fluid flows. This
model is chosen because the construction can be introduced without the
additional complications of a full fluid system. It already contains
shocks, steady states, curved-space source terms and, as discovered earlier by the authors~\cite{LeFloch2016,LeFloch2014,LeFloch2019}, it admits \emph{explicit formulas
for generalized Riemann problems}. The same strategy is expected to apply to more
complex fluid models (even without admitting such explicit formulas); in particular, the extension of our PITT method to the Schwarzschild--Euler system
is treated separately in the companion article~\cite{LeFlochXiang-2026b}. This latter model is also considered from the standpoint of well-posed theory in~\cite{BakhtinLeFloch2018,Liu2025b}. 

More generally, the strategy presented in this paper applies immediately to one-dimensional hyperbolic conservation laws and balance laws, of the general form
\be
\label{general-balance-law}
\frac{\partial u}{\partial t}
+
\frac{\partial}{\partial r}(F(u,r))
=
S(u,r),
\ee
where the flux term $F=F(u,r)$ and the source term $S=S(u,r)$ may depend on
the unknown $u=u(t,r)$ and on the spatial variable $r$. This formulation
includes standard conservation laws, for which \(S\equiv0\), as well as
geometric balance laws arising in classical or in relativistic
fluid dynamics on curved backgrounds. When its steady states and local wave curves are available, one can use the same strategy to build generalized Riemann solvers and shock-aware neural rollouts. 

The relativistic Burgers equation on a Schwarzschild background provides one of the simplest, yet non-trivial, examples where the geometry introduces non-trivial source terms and modifies the
structure of steady states. In contrast with the flat case, constant states
are no longer the natural equilibria. Instead, the relevant elementary
building blocks are steady-state profiles determined by geometric
invariants. This motivates a representation of discontinuous solutions as
piecewise steady states separated by moving or stationary shocks,
and leads to a closed description of generalized Riemann solutions in terms
of the left and right steady-state parameters.

\subsubsection{From finite-volume to front-tracking to a shock-aware architecture}

A standard numerical approach to compute shock solutions is based on a finite volume discretization
and on the construction of numerical fluxes ensuring a well-balanced property. Our approach is different, since we use this mathematical structure to design a neural architecture in which the essential physical mechanisms are built \emph{into the model} rather than inferred only from training samples.
Yet, the proposed PITT construction borrows from this standard theory of
finite-volume and geometry-preserving methods for balance laws going back to classical works such as 
\cite{vanLeer1979,Harten1983}; for developments on well-balanced schemes, we refer to the monograph~\cite{Bouchut2004} and the references therein.  Note also that a class of well-balanced schemes for the Schwarzschild--Burgers equation was introduced earlier in~\cite{lefloch2021classwellbalancedalgorithmsrelativistic}. 

The key architectural point is that the
Rankine--Hugoniot shock speed can be expressed in terms of \emph{steady-state
invariants} defined on the two sides of the moving discontinuity. Once these left and right
invariants are known, the jump condition determines the shock trajectory.
The proposed PITT architecture therefore infers the relevant invariant
parameters from the initial field and uses them as latent physical
coordinates. These coordinates are combined with symbolic equation tokens,
Fourier neural operator features, and a learned residual correction to
reconstruct the full spatio-temporal evolution. The residual component is
\emph{not} responsible for determining the shock speed; instead, it
\emph{corrects} unresolved nonlinear, geometric, and approximation effects around
the analytically prescribed shock dynamics. In this way, the
architecture \emph{learns a correction} to a mathematically meaningful prior,
instead of representing the whole discontinuous dynamics statistically.

\subsubsection{Beyond a single Riemann problem}

For general initial data, we combine the shock-aware PITT representation with a Glimm-type decomposition. The initial profile is approximated by piecewise constant or piecewise steady states, and the
evolution is represented by interacting local Riemann problems. This gives a natural way of treating multiple shocks, rarefaction waves, contact discontinuities, and wave interactions. The neural component is
thus guided by the same local wave structure that underlies classical front-tracking and random choice methods, while retaining the flexibility of operator learning. This provides the route used in the numerical
section for initial data with many discontinuities.

\subsubsection{Complex shock dynamics}

We test the proposed method on moving shocks, stationary shocks, rarefaction waves, and we investigate the
non-relativistic limit, long-time asymptotic regimes, and perturbations of
steady shocks.  Among our results, in \autoref{section=7} we observe numerically a mass-conservation displacement law (cf.~\autoref{thm:displacement}), which formally follows from a \emph{conservative form} of the balance law, and is used here as a diagnostic for our PITT predictions. 

The numerical tests are designed to separate the analytical prior from the learned
correction: steady-state invariants, generalized Riemann structure, and the
Rankine--Hugoniot shock relation are inserted into the rollout, while the scalar
model provides exact references against which learning errors can be separated from
errors introduced by an external numerical solver. The companion paper~\cite{LeFlochXiang-2026b} addresses the extension to Schwarzschild--Euler fluid flows.


\subsection{Organization of this paper}
\label{section=1--3}

The paper is organized as follows. In \autoref{section=2}, we present analytical facts about the Schwarzschild--Burgers equation used in the course of this paper; in particular, we define the specific variables supplied to the neural model. Next, in \autoref{section=3} we present the PITT architecture, with equation tokenization, Fourier neural operator encoding, parameter conditioning, physics-informed update blocks, and the Rankine--Hugoniot shock prior.  In \autoref{section=4}, we describe the practical implementation:
data generation, a high-order standard finite-volume baseline, implementation choices, and metrics. 
Next, in \autoref{section=5} we show our numerical validation, including comparisons with this baseline, ablation studies, and the flat-spacetime Burgers limit. \autoref{section=6}
treats more general piecewise steady-state data and their long-time asymptotics, which leads us, in 
\autoref{section=7}, to derive and numerically demonstrate the validity of a mass-conservation formula for
the \emph{asymptotic displacement} of perturbed steady shocks. Finally,
\autoref{section=8} gives concluding remarks, and Appendix~\ref{section=A} provides a brief dictionary of the learning terminology used in this paper. 


\section{The Schwarzschild--Burgers model and its steady-state structure}
\label{section=2}

\subsection{The model of interest} 
\label{section=2--1}

\subsubsection{Origin and scope of the model}

We begin with the Schwarzschild--Burgers model posed on the exterior domain of a Schwarzschild
black hole, which is one of the simplest, yet non-trivial, models among curved-background fluid models. This equation was derived as a relativistic Burgers
equation on curved spacetimes in
\cite{LeFlochMakhlofOkutmustur2012} and was then studied, both
analytically and numerically, in \cite{LeFloch2016,LeFloch2019,LeFloch2014,
lefloch2021classwellbalancedalgorithmsrelativistic}. In the present
article, the model is used as a controlled test case: it
retains the essential interaction between nonlinear waves, a geometric
source term, steady states, and an event horizon, while still allowing
explicit formulas that can be built into the PITT architecture. 


\subsubsection{\texorpdfstring{Conservative and balance-law forms}{Conservative and balance-law forms}}

{More precisely, we consider the relativistic Burgers equation on
the Schwarzschild exterior $r>2M$:}
\be
\label{burgers-schwarzschild}
\partial_t\!\left(\frac{u}{(1-2M/r)^2}\right)
+ \partial_r\!\left(\frac{u^2-1}{2(1-2M/r)}\right) = 0,
\ee
{where $u=u(t,r)\in[-1,1]$ denotes the radial fluid velocity, the
speed of light has been normalized to $1$, and $M>0$ is the mass of the
Schwarzschild black hole. The reduction to one space variable is obtained
under spherical symmetry. In the flat limit $M\to0$, the metric
coefficient below tends to $1$, and \eqref{burgers-schwarzschild}
reduces, up to an irrelevant additive constant in the flux, to the
standard inviscid Burgers equation
$\partial_t u+\partial_r(u^2/2)=0$. Equivalently,
\eqref{burgers-schwarzschild} can be written as the balance law}
\be
\label{burgers-2}
\frac{\partial u}{\partial t} + \frac{\partial F}{\partial r} = S,
\ee
where
\be
\label{symbols}
F = \frac{g(u^2-1)}{2}, \quad
S = \frac{2M}{r^2}(u^2-1), \quad
g(r) := 1 - \frac{2M}{r}.
\ee
{This formulation is the one used in the sequel, since it exposes
both the hyperbolic flux and the geometric source term that must be
respected by a well-balanced or physics-informed method.}


\subsubsection{\texorpdfstring{Steady-state solutions}{Steady-state solutions}}

{Setting $\partial_t u = 0$ reduces \eqref{burgers-2} to
$\partial_r F = S$. Using $F=g(u^2-1)/2$, $S=g'(u^2-1)$, and
$g'(r)=2M/r^2$, this stationary equation is equivalently written as}
\be
\label{eq:stationary-balance-ode}
\partial_r\!\left(\frac{u^2-1}{2(1-2M/r)}\right) = 0.
\ee
This ODE can be solved explicitly, giving the steady-state family
\be
\label{steady-state}
u_{K,\varsigma}(r)
     = \varsigma\sqrt{1 - K\,g(r)}
     = \varsigma\sqrt{1 - K\!\left(1 - \frac{2M}{r}\right)},
     \qquad \varsigma\in\{-1,+1\},
\ee
parameterized by a scalar $K$ and a sign branch $\varsigma$. The
Schwarzschild curvature imprints a
non-trivial radial profile on the flow. In the flat limit $M\to 0$,
\eqref{steady-state} reduces to the constant
$u = \varsigma\sqrt{1-K}$.

The parameter $K$ governs the global behavior of the steady state solution:
\begin{itemize}
\item For $0 < K \le 1$, the solution is defined for all $r > 2M$
      satisfying $|u| < 1$ (subsonic flow).
\item For $K > 1$, the solution exists only on the bounded interval
      $2M < r < 2MK/(K-1)$, where $|u|\to 1$ at the right endpoint.
\item For all $K$, as $r\to 2M^+$ we have $g\to 0$ and $|u|\to 1$,
      so the steady flow approaches the speed of light at the event horizon.
\end{itemize}

\subsubsection{\texorpdfstring{Rankine--Hugoniot condition}{Rankine-Hugoniot condition}}

For a shock wave with trajectory $t \mapsto r_s(t)$, the jump relation,
or Rankine--Hugoniot condition, associated with \eqref{burgers-2} is
\be
\label{RH-condition}
\sigma_{RH} = \frac{F_R - F_L}{u_R - u_L},
\ee
where $\sigma_{RH}:= \tfrac{d}{dt} r_s(t)$, and
$(u_L,F_L)$ and $(u_R,F_R)$ denote the left and right traces across the
discontinuity. {In this scalar setting, the source term does not
enter the jump relation itself; its effect is already encoded in the
spatial dependence of the flux and in the steady states used to define
the left and right traces.}


\subsection{A steady-state invariant}
\label{section=2--2}

\subsubsection{\texorpdfstring{Characteristic transport}{Characteristic transport}}

A key feature of \eqref{burgers-2} is that the following quantity
\be
\label{eq:K-def}
K(t,r) := \frac{1 - u^2}{g(r)}
\ee
is conserved along the characteristics of \eqref{burgers-2}. This
conservation law is the cornerstone for the closed structure of the
generalized Riemann problem and for the design of our PITT method.

\begin{proposition}[Conservation of the steady-state invariant]
\label{prop:K-conservation}
Let $u$ be a classical solution of \eqref{burgers-2} on a domain
$\mathcal{D}\subset\{t>0,\,r>2M\}$. The characteristic curves of
\eqref{burgers-2} are such that one has
\be
\label{eq:char}
\frac{dr}{dt} = g(r)\,u,
\ee
and along each characteristic one has
\be
\label{eq:K-constant-characteristic}
\frac{d}{dt}K\bigl(t,r(t)\bigr) = 0.
\ee
\end{proposition}

\begin{proof}
\bse
Along a characteristic satisfying \eqref{eq:char}, the total time
derivative of $K = (1-u^2)/g(r)$ is
\be
\label{eq:proof-K-total-derivative}
\frac{dK}{dt}
  = \frac{-2u\,\dot{u}\,g(r) - (1-u^2)\,g'(r)\,\dot{r}}{g(r)^2},
\ee
where $\dot{u} = \partial_t u + \dot{r}\,\partial_r u$ is the material
derivative. Substituting $\dot{r} = g(r)u$ and using \eqref{burgers-2}
to write $\partial_t u = S - \partial_r F$, a direct computation gives
\be
\label{eq:proof-u-dot}
\begin{aligned}
\dot{u}
  &= S - \partial_r F + g(r)u\,\partial_r u \\
  &= \frac{2M(u^2-1)}{r^2}
     - \frac{g'(r)(u^2-1)}{2}
     - g(r)u\,\partial_r u
     + g(r)u\,\partial_r u \\
  &= (u^2-1)\!\left(\frac{2M}{r^2} - \frac{g'(r)}{2}\right).
\end{aligned}
\ee
Since $g'(r) = 2M/r^2$, we obtain $\dot{u} = \tfrac{1}{2}g'(r)(u^2-1)$.
Substituting back and using $\dot{r} = g(r)u$, we find
\be
\label{eq:proof-K-zero}
\begin{aligned}
\frac{dK}{dt}
  &= \frac{1}{g(r)^2}\Bigl[
       -2u\cdot\tfrac{g'(r)}{2}(u^2-1)\cdot g(r)
       - (1-u^2)\cdot g'(r)\cdot g(r)u
     \Bigr] \notag \\
  &= \frac{g'(r)g(r)u}{g(r)^2}
     \Bigl[-(u^2-1) - (1-u^2)\Bigr] = 0. \qedhere
\end{aligned}
\ee
\ese
\end{proof}


\subsubsection{Subluminal bound and jump relation}

An immediate result is that $K$ is transported along characteristics
without change, and in particular the velocity remains bounded.

\begin{corollary}[Subluminal property of the velocity]
\label{cor:bounded}
Let $u$ be a classical solution of \eqref{burgers-2}. Then one has
$|u(t,r(t))|\leq1$ along any characteristic and for all $t\geq0$.
\end{corollary}

\begin{proof}
\bse
By Proposition~\ref{prop:K-conservation}, $K = (1-u^2)/g(r)$ is
constant along characteristics. Since $g(r) > 0$ for $r > 2M$ and
$K \geq 0$ by the initial condition $|u_0| \leq 1$, we have
$1 - u^2 = K\,g(r) \geq 0$, hence $|u| \leq 1$ for all time.
\ese
\end{proof}

The boundedness of $u$ ensures that the jump $[u] = u_R - u_L$ across
any discontinuity remains finite and nonzero as long as the two states
are distinct, so the Rankine--Hugoniot condition is well-defined. More
precisely, a discontinuity at $r_s(t)$ is a weak solution of
\eqref{burgers-2} in the sense of balance laws if and only if
\be
\label{eq:rh-balance}
\frac{dr_s}{dt}[u] = [F],
\ee
where $[u] = u_R - u_L$ and $[F] = F_R - F_L$. Dividing by $[u]\neq 0$ recovers \eqref{RH-condition}. Equation~\eqref{eq:rh-balance} is the starting point for the shock
ODE \eqref{eq:shock-ode} in Proposition~\ref{prop:closed}, stated below. 

\subsubsection{\texorpdfstring{Closed generalized Riemann dynamics}{Closed generalized Riemann dynamics}}

For initial data of the piecewise steady-state form, the invariant $K$
takes the constant value $K_L$ (resp.\ $K_R$) on every characteristic
to the left (resp.\ right) of the shock, for all time.

\begin{proposition}[Resolution of the generalized Riemann problem]
\label{prop:closed}
Consider initial data of the form
\be
\label{riemann-ic}
u(0,r) = \begin{cases}
u_{K_L,\varsigma_L}(r), & r < r_0, \\
u_{K_R,\varsigma_R}(r), & r > r_0,
\end{cases}
\ee
where $u_{K,\varsigma}$ is defined in \eqref{steady-state},
$\varsigma_L,\varsigma_R\in\{-1,+1\}$, and $r_0 > 2M$. The full
spatio-temporal evolution of this single generalized Riemann problem is
determined by $(K_L,\varsigma_L,K_R,\varsigma_R,r_0)$, and one has the
following alternatives.
\begin{enumerate}
\item[\emph{(i)}] \emph{Shock case.} If
$u_{K_L,\varsigma_L}(r_0) > u_{K_R,\varsigma_R}(r_0)$,
the shock position $r_s(t)$ is such that one has
\be
\label{eq:shock-ode}
\frac{dr_s}{dt} = \sigma_{RH}(r_s;\,K_L,K_R)
  := \frac{F(u_{K_R,\varsigma_R}(r_s)) - F(u_{K_L,\varsigma_L}(r_s))}
         {u_{K_R,\varsigma_R}(r_s) - u_{K_L,\varsigma_L}(r_s)},
\ee
and $u(t,r) = u_{K_L,\varsigma_L}(r)$ for $r < r_s(t)$, $u_{K_R,\varsigma_R}(r)$ for
$r > r_s(t)$.
{
\item[\emph{(ii)}] \emph{Rarefaction case.} If
$u_{K_L,\varsigma_L}(r_0) < u_{K_R,\varsigma_R}(r_0)$,
characteristics diverge and the rarefaction fan
expands between two boundary curves $r_L(t) \leq r_R(t)$.
For the fan boundaries one has
\be
\label{eq:char-fan}
\frac{dr}{dt} = g(r)\,u_{K,\varsigma}(r)
= \left(1 - \frac{2M}{r}\right)u_{K,\varsigma}(r),
\ee
with $r_L(t)$ carrying $(K_L,\varsigma_L)$ and $r_R(t)$ carrying
$(K_R,\varsigma_R)$, both initialized at $r_0$. Inside the fan
$r_L(t) < r < r_R(t)$, the solution takes the form
\be
\label{eq:rarefaction-intermediate-profile}
u(t,r) = \varsigma(t,r)\sqrt{1 - K(t,r)\,g(r)},
\ee
where $K(t,r)$, together with the sign branch, is implicitly determined
by the condition that the characteristic satisfying \eqref{eq:char-fan}
passes through the point $(t,r)$. Outside the fan,
$u(t,r) = u_{K_L,\varsigma_L}(r)$ for $r < r_L(t)$ and
$u(t,r) = u_{K_R,\varsigma_R}(r)$ for $r > r_R(t)$.}
\end{enumerate}
\end{proposition}

\begin{proof}
\bse
By Proposition~\ref{prop:K-conservation}, $K \equiv K_L$ on all
characteristics to the left of the shock, and $K \equiv K_R$ on the
right. The left and right traces at the shock are therefore
$u_{K_L,\varsigma_L}(r_s(t))$ and
$u_{K_R,\varsigma_R}(r_s(t))$ for all $t$, and the
Rankine--Hugoniot condition \eqref{RH-condition} reduces to the
autonomous ODE \eqref{eq:shock-ode}. Existence and uniqueness of the
shock trajectory follow from the Cauchy--Lipschitz theorem, since
$\sigma_{RH}$ is smooth in $r_s$ for $r_s > 2M$, away from the degenerate case in which the two traces coincide.

In the rarefaction case, by Proposition~\ref{prop:K-conservation},
$K$ is constant along each characteristic. If the initial traces satisfy
\(u_{K_L,\varsigma_L}(r_0) < u_{K_R,\varsigma_R}(r_0)\), then, since
$g(r_0)>0$, the corresponding characteristic speeds are
ordered in the same way. Hence characteristics emanating from $r_0$
diverge and a rarefaction fan opens. Each characteristic inside the fan
carries a fixed invariant and a fixed sign branch, which determines the
intermediate profile implicitly. Outside the fan, the same conservation
law gives the two original steady states.
\ese
\end{proof}

The shock trajectory given by the ODE
\eqref{eq:shock-ode} is precisely the quantity approximated by PITT:
the analytic term $\sigma_{RH}^{(t)}$ in the rollout \eqref{eq:rh-speed}
evaluates the right-hand side of \eqref{eq:shock-ode} at the current
predicted shock position, and the learned correction
$\Delta\sigma_{NN}^{(t)}$ accounts for residual transient effects not represented by the steady-state prior.


\subsection{Approximation of general solutions}
\label{section=2--3}

\subsubsection{\texorpdfstring{Approximation on a computational interval}{Approximation on a computational interval}}

{The previous construction is exact for a single generalized
Riemann problem. For more general initial data, the relevant point is that
piecewise steady-state profiles form a flexible local approximation class
on any computational interval separated from the horizon. 

\begin{proposition}[Approximation by piecewise steady-state profiles]
\label{lem:density}
{Let $I=[r_{\min},r_{\max}]$ with $2M<r_{\min}<r_{\max}<+\infty$.
Then one has the following density property: the set of piecewise
steady-state functions of the form}
\be
\label{eq:piecewise-steady-approx-class}
{
u(r) = \sum_{i=1}^{N} \varsigma^i \sqrt{\max(0,\, 1 - K^i g(r))}\,
       \mathbf{1}_{[r_{i-1}, r_i)}(r),
\quad K^i \geq 0,\; \varsigma^i \in \{-1,+1\},
}
\ee
{is dense in the set of $BV(I)$ profiles satisfying $|u|\leq1$, with
respect to the $L^1(I)$ norm.}
\end{proposition}

\begin{proof}
\bse
{Since $g$ is smooth and bounded away from zero on $I$, the map}
\be
\label{eq:proof-density-local-map}
{
(K,\varsigma,r)\mapsto \varsigma\sqrt{\max(0,1-Kg(r))}
}
\ee
{is continuous on compact subsets relevant to $|u|\leq1$.
Given $\varepsilon>0$, choose
a partition $r_{\min}=r_0<r_1<\cdots<r_N=r_{\max}$ fine enough that the
oscillation of both $u_0$ and $g$ on each subinterval is small, except on
a set of total $L^1$ contribution below $\varepsilon$. On
$[r_{i-1},r_i)$, choose a point $\bar r_i$ and a representative value
$\bar u_i$ of $u_0$, then set
\(\varsigma^i=\operatorname{sign}(\bar u_i)\) and
\(K^i=(1-\bar u_i^2)/g(\bar r_i)\).
The associated steady profile agrees with $\bar u_i$ at $\bar r_i$ and
remains close to it on the whole subinterval when the mesh is small. The
standard approximation of bounded-variation functions by piecewise constants therefore
gives an $L^1(I)$ error bounded by $C\varepsilon$, where $C$ depends only
on $I$, $M$, and the total variation of $u_0$.}
\ese
\end{proof}

\subsubsection{\texorpdfstring{From local Riemann problems to front tracking}{From local Riemann problems to front tracking}}

Proposition~\ref{lem:density} provides the approximation-theoretic foundation
for treating general initial data. Given any $u_0 \in BV$, we recover
the conserved invariant field $K(r) = (1-u_0^2)/g(r)$ pointwise via
Proposition~\ref{prop:K-conservation}, approximate it by a piecewise
constant function with finitely many jumps, and decompose the resulting
problem into independent generalized Riemann subproblems, each governed
by adjacent steady branches $(K^i,\varsigma^i)$ and
$(K^{i+1},\varsigma^{i+1})$. 

We adopt a front-tracking strategy: the
position of each shock is tracked explicitly via the
Rankine--Hugoniot ODE~\eqref{eq:shock-ode}, while the smooth field
between discontinuities is reconstructed from the piecewise
steady-state profiles $u_{K^i,\varsigma^i}$. Rarefaction waves are
represented by characteristic fans, while the learned residual correction
in the field reconstruction \eqref{field-reconstruction} supplies a
compact neural representation of the smooth fan profile on the numerical
grid. The exact front-tracking description is valid up to the first wave
interaction time $t_{\mathrm{coll}}$; beyond that time, the same
decomposition becomes a structured approximation used by the PITT rollout,
in the spirit of the random choice and front-tracking constructions
\cite{Dafermos2010,Glimm1965,LeFloch2002}.}


\section{A shock-aware PITT method for Schwarzschild--Burgers flows}
\label{section=3}

\subsection{Equation tokenization}
\label{section=3--1}

{ 

\subsubsection{Brief presentation}

We now describe the proposed PITT method, developed in this paper for the
relativistic Schwarzschild--Burgers equation. In particular, our aim is to
explain how the usual ingredients of a shock-capturing method are
translated into a neural architecture. In a standard finite-volume scheme,
one specifies the balance-law form of the equation, chooses a numerical
flux or a Riemann solver, and advances cell averages in time. Here, the
same physical information is used, but it is supplied to the neural model
through symbolic equation tokens and through an \emph{explicit
Rankine--Hugoniot shock prior}. {The construction uses the PITT idea of
equation tokens~\cite{Lorsung2024}, Transformer attention
\cite{Vaswani2017,Cao2021}, and neural-operator encodings
\cite{Li2021,Lu2021,Kovachki2023}.}

We focus on the scalar balance law \eqref{burgers-2}.
The proposed architecture has four main components: an \emph{equation
tokenizer}, an \emph{FNO field encoder}, a \emph{parameter encoder}, and a
stack of \emph{physics-informed update blocks}. These components are
followed by an analytic shock rollout based on the Rankine--Hugoniot
condition. {Thus, the network is not asked to discover the shock
speed from training data. In the present method, the shock motion is
constrained by the jump condition from standard hyperbolic theory, while
the neural part learns how to reconstruct the surrounding smooth or
piecewise smooth solution and how to correct unresolved geometric and
nonlinear effects.}

	{The architecture should therefore be viewed as a \textbf{hybrid between a
structured solver and a neural operator}. The closed formulas from
\autoref{section=2} provide the wave variables and the admissible
shock motion, while the learned part supplies a compact representation of
the spatial field and of residual effects that are not explicitly encoded
by the steady-state Ansatz.}

The first step is to convert the governing equation into a form that can be
processed by a Transformer. This is what we call \emph{equation
tokenization}. The idea is similar to writing the PDE as a symbolic
sentence. The unknown, the independent variables, the flux, the source
term, the physical parameters, and the Rankine--Hugoniot relation are all
represented by elementary symbols. These symbols are then mapped to integer
labels, called \emph{token IDs}, and processed by a Transformer encoder. In this
way, the neural model is not only given samples of the solution field; it
is also given an \emph{explicit description of the equation} that generated the
data.

This construction should be compared with a classical finite-volume solver
as follows. A finite-volume method uses the formula for the flux and the
source term \emph{directly} in the update rule. In a PITT method, the formula for the flux,
the source term, and the shock relation are instead encoded as a \emph{symbolic
sequence}. The Transformer converts this sequence into a vector
representation, called the physics embedding. This embedding constrains the subsequent neural update blocks. Thus, \emph{the architecture is made aware of the governing balance law} before it predicts the evolution. In particular, equation tokenization is not used here as a symbolic
parser. Its role is to expose, in a uniform learnable format, the
mathematical ingredients that distinguish one balance law, one geometry,
or one wave configuration from another. 

\begin{table}[H]
\centering
\small
\caption{Token vocabulary. $\langle\texttt{PAD}\rangle$: padding;
$\langle\texttt{SEP}\rangle$: section separator;
$\langle\texttt{NUM}\rangle$: numeric placeholder.}
\label{tab:vocab}

\renewcommand{\arraystretch}{1.08}
\begin{tabular}{@{}c@{\hspace{1.8em}}c@{\hspace{1.8em}}c@{}}

\begin{tabular}{c|c}
\toprule
Token & ID \\ \midrule
$\langle\texttt{PAD}\rangle$ & 0 \\
$\langle\texttt{SEP}\rangle$ & 1 \\
$\langle\texttt{NUM}\rangle$ & 2 \\
$\partial$ & 3 \\
$u$ & 4 \\
$t$ & 5 \\
$r$ & 6 \\
$+$ & 7 \\
$-$ & 8 \\
$*$ & 9 \\
\bottomrule
\end{tabular}

&

\begin{tabular}{c|c}
\toprule
Token & ID \\ \midrule
$/$ & 10 \\
$=$ & 11 \\
$F$ & 12 \\
$S$ & 13 \\
$M$ & 14 \\
$K$ & 15 \\
$($ & 16 \\
$)$ & 17 \\
${}^2$ & 18 \\
$0$--$9$ & 19--28 \\
\bottomrule
\end{tabular}

&

\begin{tabular}{c|c}
\toprule
Token & ID \\ \midrule
$.$ & 29 \\
$\&$ & 30 \\
$R$ & 31 \\
$H$ & 32 \\
$:$ & 33 \\
$s$ & 34 \\
$[$ & 35 \\
$]$ & 36 \\
$,$ & 37 \\
$\varsigma$ & 38 \\
\bottomrule
\end{tabular}

\end{tabular}
\end{table}
%
\subsubsection{Terminology and notation}

It is convenient to summarize the relevant vocabulary $\mathcal{V}$ in
Table~\ref{tab:vocab}. By design, it contains $|\mathcal{V}|=39$ tokens,
covering
\bei
\item special tokens denoted by
$\langle\texttt{PAD}\rangle$,
$\langle\texttt{SEP}\rangle$, and
$\langle\texttt{NUM}\rangle$;

\item differential operators, variables, arithmetic symbols, and physical
parameters;

\item digits and symbols associated with the sign branches and the
Rankine--Hugoniot condition.
\eei
The token $\langle\texttt{PAD}\rangle$ is used to fill short sequences up
to a fixed length, $\langle\texttt{SEP}\rangle$ separates logical parts of
the symbolic description, and $\langle\texttt{NUM}\rangle$ may be used as a
placeholder for numerical values. In the implementation described below,
numerical constants are also tokenized digit by digit, which gives the
network access to the actual parameter values.

Given the balance law \eqref{burgers-2} and the Rankine--Hugoniot condition
\eqref{RH-condition}, each component is serialized into a token sequence.
For example, with $K_L=0.200$, $K_R=0.400$, and $M=1.000$, one obtains
\be
\label{eq:token-sequence-example}
\begin{aligned}
\texttt{[}\langle\texttt{SEP}\rangle,\;
  &\partial, u, /, \partial, t, +, \partial, F, /, \partial, r, =, S, \&,\\
  &F, =, (, 1, -, 2, *, M, /, r, ), *, (, u, {}^2, -, 1, ), /, 2, \&,\\
  &S, =, (, 2, *, M, /, r, {}^2, ), *, (, u, {}^2, -, 1, ), \&,\\
  &M, =, 1, ., 0, 0, 0, \&,\quad
   K, =, 0, ., 2, 0, 0, \&,\quad
   \varsigma, =, +, \&,\\
  &K, =, 0, ., 4, 0, 0, \&,\quad
   \varsigma, =, -, \&,\quad
   R, H, :, s, =, [, F, ], /, [, u, ],\;
  \langle\texttt{SEP}\rangle\texttt{]}.
\end{aligned}
\ee
Here the first line encodes the abstract balance-law structure, the next
two lines encode the flux and the geometric source term, the following
entries encode the physical parameters, invariant parameters, and sign
branches, and the final part encodes the Rankine--Hugoniot shock-speed
relation. The symbol $\&$
separates equations within the same logical block.

Numeric values are tokenized digit by digit to three decimal places; for
instance, 
\be
\label{eq:digit-tokenization}
0.200\mapsto\texttt{[`0', `.', `2', `0', `0']}.
\ee
This is a simple way of keeping the token vocabulary small while still
allowing the Transformer to distinguish different parameter values. The
full token sequence is then padded or truncated to a fixed length $L$, so
that all samples in a batch have the same input size.

Each token is mapped to a learnable $d_{\mathrm{model}}$-dimensional
embedding. A learnable positional encoding of the same dimension is added
in order to preserve the order of the symbols, since the meaning of a
formula depends on the order in which its tokens appear. The resulting
sequence $\mathbf{T}\in\mathbb{N}^{B\times L}$, where $B$ denotes the
batch size, is processed by a Transformer encoder
with multi-head self-attention. Padding tokens, which have ID $0$, are
excluded from the attention computation by an additive mask. Thus, only the
meaningful part of the symbolic equation contributes to the representation.

The output of the Transformer encoder is then averaged over the sequence
length to produce a global physics embedding
\(\mathbf{e}_{\mathrm{phys}}\in
\mathbb{R}^{B\times d_{\mathrm{model}}}\).
This vector should be interpreted as a learned representation of the
governing balance law, the source term, the relevant parameters, and the
shock relation. It is passed to the subsequent field encoder and update
blocks, so that the neural evolution is conditioned by the equation itself
rather than only by the sampled initial data.}

 
\subsection{Encoding the initial field and the geometry with an FNO}
\label{section=3--2}

{ 

\subsubsection{\texorpdfstring{Input channels}{Input channels}}

The equation tokenizer described above produces a global representation of
the governing balance law. We now describe the second input to the PITT
architecture, namely the \emph{sampled initial field}. In a classical
finite-volume method, the initial data are represented by cell averages or
nodal values on a spatial grid. Here, we use the same type of sampled
information, but we lift it to a latent representation by means of a Fourier Neural Operator (FNO)~\cite{Li2021}.

The encoder receives the initial profile
$\mathbf{U}_{\mathrm{in}}\in\mathbb{R}^{B\times N_r}$, where $B$ is the
batch size and $N_r$ is the number of grid points. This array represents
the initial velocity field $u(0,\cdot)$ in \eqref{burgers-2} evaluated on a
uniform radial grid. Since the relativistic Schwarzschild--Burgers equation
depends explicitly on the background geometry, the metric coefficient
\(g(r)=1-2M/r\) is also provided to the encoder. In practice, $g(r)$ is concatenated with
the initial field as an additional channel. Thus, the encoder does not see
only the fluid state; it also receives the geometric information that
enters the flux, the source term, and the steady-state structure.

\subsubsection{\texorpdfstring{Fourier lifting and projection}{Fourier lifting and projection}}

The FNO block acts as a global spatial feature extractor. Given a latent
field $v$, it applies a Fourier transform in the radial variable, keeps a
prescribed number $\kappa$ of low-frequency modes, multiplies these modes by
learnable complex weights, and transforms the result back to physical
space. At the level of one Fourier layer, this operation can be written as
\be
\label{eq:fno-layer}
(\mathcal{K}v)(r)
=
\mathcal{F}^{-1}
\Bigl(
R_\kappa \cdot \mathcal{F}v
\Bigr)(r),
\ee
where $\mathcal{F}$ denotes the discrete Fourier transform in the spatial
variable and \(R_\kappa\in
\mathbb{C}^{\kappa\times C_{\mathrm{in}}\times C_{\mathrm{out}}}\)
is a learnable complex weight tensor. The parameter $\kappa$ controls the
number of retained Fourier modes, while $C_{\mathrm{in}}$ and
$C_{\mathrm{out}}$ denote the input and output channel dimensions.

The output of the FNO is then projected by a linear map
$\mathbf{W}_{\mathrm{proj}}$ to the common latent dimension used by the
Transformer-based part of the architecture. This gives the initial latent
state
\(\mathbf{V}^{(0)}\in
\mathbb{R}^{B\times N_r\times d_{\mathrm{model}}}\).
The index $r$ is still spatial: for each grid point, the model stores a
$d_{\mathrm{model}}$-dimensional feature vector. This latent field is the
neural analogue of the initial discrete state in a numerical scheme, but
with additional channels \emph{encoding nonlocal spatial information and
geometric dependence.}

\subsubsection{Scope of this step}

The use of an FNO should not be interpreted as a replacement for the
shock-capturing mechanism. Fourier representations alone \emph{may introduce
oscillations or smoothing} near discontinuities. In the present method, the
FNO is used to encode the global shape of the initial field and its
interaction with the Schwarzschild geometry, while the discontinuity motion
is governed separately by the Rankine--Hugoniot prior. This division of
roles is essential: the FNO provides a resolution-flexible field
representation, whereas the shock prior supplies the local jump information
needed to propagate discontinuities accurately.

}


\subsection{Parameter encoder and global conditioning}
\label{section=3--3}

\subsubsection{\texorpdfstring{Physical parameters}{Physical parameters}}

{In the shock-tracking regime, the additional continuous
parameters supplied to the architecture are the black-hole mass $M$ and
the initial shock position $r_s^{(0)}$. The invariant parameters
$(K_L,\varsigma_L,K_R,\varsigma_R)$ are already represented in the
equation-token stream, because they determine the left and right
steady-state branches. The pair $(M,r_s^{(0)})$ is encoded by a two-layer
	MLP with GELU activation:}
	\be
	\label{eq:parameter-encoder}
	\mathbf{e}_{\mathrm{param}}
	  = \mathrm{MLP}_{\mathrm{param}}\!\bigl([M,\,r_s^{(0)}]\bigr)
	  \in\mathbb{R}^{B\times d_{\mathrm{model}}}.
	\ee

\subsubsection{\texorpdfstring{Global and temporal conditioning}{Global and temporal conditioning}}

The base conditioning vector combines the physics and parameter
embeddings, so that
\(\mathbf{c}^{(0)}=\mathbf{e}_{\mathrm{phys}}+\mathbf{e}_{\mathrm{param}}\),
where $\mathbf{e}_{\mathrm{phys}}$ is the global physics embedding from
\autoref{section=3--1}. At each rollout step $t$, a learnable time
embedding $\mathbf{e}_t\in\mathbb{R}^{d_{\mathrm{model}}}$ is added,
giving \(\mathbf{c}^{(t)}=\mathbf{c}^{(0)}+\mathbf{e}_t\). This ensures that both the physical structure of the PDE and the current
temporal context jointly modulate every layer of the latent dynamics. In
the formulas below, $\mathbf{c}^{(t)}$ is broadcast over the spatial grid
whenever it is added to a tensor in $\mathbb{R}^{B\times N_r\times
d_{\mathrm{model}}}$.


\subsection{Physics-informed update blocks}
\label{section=3--4}

{The latent field state
$\mathbf{V}\in\mathbb{R}^{B\times N_r\times d_{\mathrm{model}}}$ is
iteratively refined by $N_{\mathrm{blocks}}$ stacked update blocks. Each
block performs cross-attention between the spatial latent state and a
compressed summary of the symbolic token sequence, followed by a residual
MLP update.}

\subsubsection{\texorpdfstring{Compressed token summaries}{Compressed token summaries}}

{The token sequence}
$\mathbf{T}_{\mathrm{seq}}\in\mathbb{R}^{B\times L\times d_{\mathrm{model}}}$
is compressed into $n_s$ summary tokens via chunk mean-pooling:
\be
\label{eq:token-chunk-pooling}
\widetilde{\mathbf{T}}_i
  = \frac{1}{L/n_s}\sum_{j=(i-1)L/n_s}^{iL/n_s - 1}
    \mathbf{T}_{\mathrm{seq},j},
\quad \widetilde{\mathbf{T}}\in\mathbb{R}^{B\times n_s\times d_{\mathrm{model}}},
\ee
{preserving the structural segments of the token sequence: PDE
body, flux $F$, source $S$, parameters, and Rankine--Hugoniot relation.
This reduces the cross-attention cost from $\mathcal{O}(N_r L)$ to
$\mathcal{O}(n_s N_r)$ while keeping the equation-level information
available to every spatial point.}

\subsubsection{\texorpdfstring{Cross-attention update}{Cross-attention update}}

Queries are projected from the spatial latent state, while keys
and values are projected from the compressed token summary:
\be
\label{eq:attention-projections}
Q = \mathbf{W}_Q\mathbf{V},\qquad
K_{\mathrm{tok}} = \mathbf{W}_K\widetilde{\mathbf{T}},\qquad
V_{\mathrm{tok}} = \mathbf{W}_V\widetilde{\mathbf{T}}.
\ee
With these projections, we have 
$Q\in\mathbb{R}^{B\times N_r\times d_{\mathrm{model}}}$ and
$K_{\mathrm{tok}},V_{\mathrm{tok}}\in
\mathbb{R}^{B\times n_s\times d_{\mathrm{model}}}$. The attention map is
\be
\label{eq:attention-map}
A = \mathrm{softmax}\!\left(
      \frac{QK_{\mathrm{tok}}^{\top}}{\sqrt{d_{\mathrm{model}}}}
    \right)
    \in \mathbb{R}^{B\times N_r\times n_s},
\ee
and the latent field is updated by
\be
\label{eq:latent-update}
\mathbf{O}=A V_{\mathrm{tok}},
\qquad
\mathbf{V}
  \leftarrow
  \mathbf{V}
  + \alpha\;\mathrm{MLP}_{\mathrm{update}}\!\bigl(
    \mathbf{O} + \mathbf{c}^{(t)}\bigr),
\ee
{where $\mathbf{c}^{(t)}$ is broadcast over the radial grid. The
factor $\alpha$ is a fixed step size that keeps the learned correction
small relative to the analytic shock and steady-state priors.}


\subsection{RH-guided shock evolution and field reconstruction}
\label{section=3--5}

\subsubsection{\texorpdfstring{Shock trajectory}{Shock trajectory}}

{In the shock regime, the discontinuity is not reconstructed from the grid values alone. Instead, its position is an explicit variable. At
each rollout step $t$, the shock position is advanced by}
\be
\label{eq:shock-position-update}
r_s^{(t+1)} = r_s^{(t)}
  + \bigl(\sigma_{RH}^{(t)} + \Delta\sigma_{NN}^{(t)}\bigr)\,\Delta t,
\ee
{where $\Delta t$ is the rollout time step. The dominant term}
$\sigma_{RH}^{(t)}$ evaluates the right-hand side of \eqref{eq:shock-ode}
at the current predicted shock position:
\be
\label{eq:rh-speed}
\sigma_{RH}^{(t)}
  = \frac{F_R(r_s^{(t)}) - F_L(r_s^{(t)})}
         {u_R(r_s^{(t)}) - u_L(r_s^{(t)})},
\ee
{with
$u_L(r)=u_{K_L,\varsigma_L}(r)$ and
$u_R(r)=u_{K_R,\varsigma_R}(r)$. The neural correction}
\be
\label{eq:neural-shock-correction}
\Delta\sigma_{NN}^{(t)}
  = \epsilon\;\tanh\!\bigl(\mathbf{W}_{\mathrm{shock}}\,\bar{\mathbf{V}}^{(t)}\bigr)
    \cdot\gamma(K_L,K_R)
\ee
is gated by
\be
\label{eq:shock-gate}
\gamma(K_L,K_R)
  = \sigma_{\mathrm{sig}}\!\bigl(\lambda\,(|K_L-K_R|-\delta)\bigr),
\ee
where $\bar{\mathbf{V}}^{(t)} = \frac{1}{N_r}\sum_i\mathbf{V}_i^{(t)}$
is the spatial mean of the latent state, $\epsilon$ controls the correction
magnitude, and $(\lambda,\delta)$ are gate sharpness and threshold
hyperparameters. The gate $\gamma$ suppresses the neural correction when
$|K_L-K_R|$ is small, i.e., near steady-state configurations where
$\sigma_{RH}\approx 0$. {In exact generalized Riemann data, this
correction is expected to remain small; its role is to absorb
discretization error, finite-resolution effects, and residual mismatch in
the learned latent dynamics. The shock position is clamped to
$r_s\in(2M+\varepsilon,\,r_{\max}-\varepsilon)$ at every step.}

\subsubsection{\texorpdfstring{Field reconstruction and rarefaction mode}{Field reconstruction and rarefaction mode}}

{In the shock regime, given $r_s^{(t)}$, the full spatial field is
reconstructed as}
\be
\label{field-reconstruction}
u^{(t)}(r)
  = u_L(r)\,s(r) + u_R(r)\,\bigl(1-s(r)\bigr)
  + \mu\;\tanh\!\bigl(\mathbf{W}_{\mathrm{field}}\,\mathbf{V}^{(t)}\bigr),
\ee
where $s(r) = \sigma_{\mathrm{sig}}\!\bigl((r_s^{(t)}-r)/w\bigr)$ is a
smooth step function with transition width $w$. The first two terms
constitute the \emph{physics prior}: a sharp but differentiable
approximation of the exact piecewise steady-state solution parameterized
by $(K_L,\varsigma_L,K_R,\varsigma_R)$. {The final term is a learned
residual correction, bounded in magnitude by $\mu$, that accounts for
deviations from the steady-state profiles during transient dynamics.}
{In the rarefaction regime, no shock position is evolved. The two
fan boundaries are instead advanced by the characteristic ODE
\eqref{eq:char-fan}, and the decoder uses the same latent field
$\mathbf{V}^{(t)}$ to reconstruct the smooth transition inside the fan.
Thus, the architecture separates the treatment of discontinuity motion
from the reconstruction of smooth or rarefaction regions.}


\subsection{Architecture summary}
\label{section=3--6}

\subsubsection{\texorpdfstring{Forward pass}{Forward pass}}

The complete PITT forward pass proceeds as follows:
\begin{enumerate}
\item \textbf{Encode.} FNO encodes $(u_0,\,g(r))\to\mathbf{V}^{(0)}$;
  the Transformer encodes the token sequence $\to\mathbf{e}_{\mathrm{phys}}$;
  the MLP encodes $(M,\,r_s^{(0)})\to\mathbf{e}_{\mathrm{param}}$.

\item \textbf{Condition.} Form
  $\mathbf{c}^{(t)} = \mathbf{e}_{\mathrm{phys}} + \mathbf{e}_{\mathrm{param}}
  + \mathbf{e}_t$ for each rollout step $t$.

\item \textbf{Rollout.} For $t = 0,\ldots,T-1$:
  \begin{enumerate}
  \item Apply $N_{\mathrm{blocks}}$ update blocks to refine $\mathbf{V}^{(t)}$.
  \item If the initial data are
        of shock type, that is,
        \(u_{K_L,\varsigma_L}(r_0) > u_{K_R,\varsigma_R}(r_0)\):
        compute $\sigma_{RH}^{(t)}$
        from $(K_L,\varsigma_L,K_R,\varsigma_R,r_s^{(t)})$
        via \eqref{eq:rh-speed}, compute
        $\Delta\sigma_{NN}^{(t)}$ from $\bar{\mathbf{V}}^{(t)}$, and
        advance $r_s^{(t+1)}$.
  \item If the initial data are
        of rarefaction type, that is,
        \(u_{K_L,\varsigma_L}(r_0) < u_{K_R,\varsigma_R}(r_0)\):
        no shock position is tracked.
        The fan boundaries $r_L^{(t)}$ and $r_R^{(t)}$ are advanced by
        integrating the characteristic ODE \eqref{eq:char-fan} with
        $u = u_{K_L,\varsigma_L}$ and
        $u = u_{K_R,\varsigma_R}$, respectively.
  \item Reconstruct the full field $u^{(t)}(r)$ via
        \eqref{field-reconstruction}, where the bounded field residual
        accounts for the smooth fan profile in the rarefaction case and
        transient deviations from the steady-state prior in the shock case.
  \end{enumerate}

\item \textbf{Output.} Return $\{u^{(t)}\}_{t=0}^{T-1}$ and
  $\{r_s^{(t)}\}_{t=0}^{T-1}$ (shock case) or
  $\{r_L^{(t)}, r_R^{(t)}\}_{t=0}^{T-1}$ (rarefaction case).
\end{enumerate}

\subsubsection{\texorpdfstring{Regime-dependent outputs}{Regime-dependent outputs}}

{The output variables therefore depend on the wave regime. For a
shock sample, the model predicts both a field and a shock trajectory, and
the Rankine--Hugoniot prior is directly active. For a rarefaction sample,
the explicit geometric evolution is carried by the two fan boundaries,
while the field decoder reconstructs the smooth fan profile. This
distinction is important for the numerical tests below: the same learned
latent dynamics is used in both regimes, but the analytical wave prior is
adapted to the corresponding entropy solution.}


\section{Numerical methodology and implementation}
\label{section=4}  

\subsection{Objective of this section}
\label{section=4--1}

{
\subsubsection{Scope of the numerical protocol}

In this section we test, for the Schwarzschild--Burgers model introduced
in \autoref{section=2}, the PITT architecture developed in
\autoref{section=3}. We first specify the data generation procedure
and the finite-volume reference solver
(Sections~\ref{section=4--2}--\ref{section=4--3}), and then record the
implementation choices and evaluation criteria
(Sections~\ref{section=4--4}--\ref{section=4--5}).}

{
\subsubsection{Connection with the analytical wave structure}

The numerical tests are organized around the three elementary regimes
identified in Proposition~\ref{prop:closed}: moving shocks, steady shocks,
and rarefaction waves. This organization allows us to
separate the effect of the Rankine--Hugoniot prior, the effect of the
equation tokenizer, and the residual neural correction. The ablation study
in \autoref{section=5--2} isolates these contributions, while
\autoref{section=5--3} checks that the same architecture is
consistent with the flat-spacetime limit $M\to0$.}


\subsection{Dataset and data generation}
\label{section=4--2}

\subsubsection{Analytic data generation}
We use the exact Riemann solver to generate training and validation data.
For each sample, the black hole mass $M$, steady-state parameters
$(K_L, K_R)$, flow direction signs
$(\varsigma_L,\varsigma_R)\in\{-1,+1\}^2$, initial discontinuity position
$r_0$, and solution type are drawn randomly. The initial field is the
piecewise steady-state profile \eqref{eq:ic}, where $\varsigma_{L,R}$
encodes the flow direction on each side:
\be
\label{eq:ic}
u(0, r) = \begin{cases}
\varsigma_L\,\sqrt{\max(0,\,1 - K_L\,g(r))}, & 2M < r < r_0, \\
\varsigma_R\,\sqrt{\max(0,\,1 - K_R\,g(r))}, & r > r_0.
\end{cases}
\ee
The regimes are selected by comparing the two traces
$u_L^0=u_{K_L,\varsigma_L}(r_0)$ and
$u_R^0=u_{K_R,\varsigma_R}(r_0)$. Moving shocks correspond to
$u_L^0>u_R^0$, steady shocks to $K_L=K_R$ and
$\varsigma_L=-\varsigma_R$, and rarefaction waves to $u_L^0<u_R^0$.
Full trajectories are generated semi-analytically: the steady branches are
evaluated from the closed formulas, while a fourth-order Runge--Kutta
method with step $\Delta t$ is used for \eqref{eq:shock-ode} in shock cases
and for the rarefaction fan boundaries in rarefaction cases. The spatial domain uses a uniform grid of $N_r$ points
over $[r_{\min},r_{\max}]$ with $r_{\min}=2M+\varepsilon$. Each trajectory
covers $t\in[0,T]$ with time step $\Delta t$, yielding $T/\Delta t$ target
frames. Solution types are sampled at fixed proportions: $40\%$ moving
shocks, $30\%$ steady shocks, and $30\%$ rarefaction waves.

{
\subsubsection{Reproducibility parameters}

Unless otherwise stated, the random seed is fixed to $2026$, the horizon
offset is $\varepsilon=10^{-3}$, and the computational domain is
\([2M+\varepsilon,15]\). We use $N_r=500$ grid points and $100$ output
times with $\Delta t=10^{-2}$. For moving shocks and rarefactions, $M$ and
$(K_L,K_R)$ are sampled from the ranges in Table~\ref{tab:dataset},
$r_0$ is sampled uniformly from
\([2M+0.5,\,12]\), and the four sign choices are first sampled uniformly
and then rejected unless they produce the desired regime
$u_L^0>u_R^0$ or $u_L^0<u_R^0$. For steady shocks we set
$K_L=K_R=:K$, take $\varsigma_L=+1$ and $\varsigma_R=-1$, and sample
$K$ and $r_0$ from the same ranges. Samples for which the square root in
\eqref{eq:ic} is truncated on more than $5\%$ of the grid are rejected,
so that the training set is not dominated by near-vacuum clipped states.}

\subsubsection{Code and data generation}

The data are generated directly from the parameter ranges, rejection rules,
and ODE solvers described in this section. For reproducibility, the
implementation used for the figures records the random seed, the accepted
parameter tuples, the time grid, the finite-volume CFL history, and the
trained model parameters. Since the sampling includes rejection
steps, the accepted parameter tuples are part of the reproducibility data; a
random seed alone is not a complete description if implementation details
are changed. The scripts producing the training set, test set,
tables, and figures will be archived with the submitted version of the
article, so that the numerical results can be regenerated from the
information specified here.

\subsubsection{Dataset splits}
The dataset is divided into three splits summarized in
Table~\ref{tab:dataset}. Training and validation sets use analytic
generation. {The test set is generated with fresh random
parameters. For every test trajectory we compute the analytic reference
solution and, on the same grid and time interval, the finite-volume
solution described below. All reported MAE values are measured against the
analytic reference; the finite-volume computation is reported as the
standard numerical baseline.} The validation and test sets cover parameter
ranges partially out-of-distribution relative to training, assessing
generalization to unseen physical configurations.

\begin{table}[ht]
\centering
\caption{Dataset splits. Train and validation use analytic generation.
For the test set, both the analytic reference and the numerical baseline
are computed on the same trajectories.}
\begin{tabular}{lcccc}
\hline
Split & Generation & Trajectories & $M$ range & $(K_L,K_R)$ range \\ \hline
Train      & Analytic  & 300 & $[0.8,\,1.2]$ & $[0.1,\,0.8]$ \\
Validation & Analytic  &  60 & $[1.2,\,1.5]$ & $[0.4,\,0.6]$ \\
Test       & Analytic + numerical &  30 & $[0.8,\,1.5]$ & $[0.1,\,0.8]$ \\ \hline
\end{tabular}
\label{tab:dataset}
\end{table}


\subsection{Baseline}
\label{section=4--3}

As a reference solver we use a finite-volume scheme combining third-order
strong-stability-preserving Runge--Kutta (SSP-RK3) time integration,
MUSCL reconstruction with the van Leer limiter~\cite{vanLeer1979}, and an
HLL numerical flux~\cite{Harten1983}.
The scheme solves \eqref{burgers-2}
on a uniform spatial grid with CFL-controlled adaptive time stepping and
outflow boundary conditions at the right boundary. This choice is
consistent with the geometry-preserving finite-volume strategy developed
for balance laws with sources
\cite{CeylanLeFlochOkutmustur2018,LeFlochMakhlof2014,DongLeFloch2019,
GiesselmannLeFloch2020}; we verify below that the resulting scheme is in
fact well-balanced to machine precision outside the black hole horizon.

{
\subsubsection{Finite-volume discretization}

The baseline is a standard shock-capturing finite-volume method. Its role
is to provide a conservative reference computation on the same spatial
domain, with numerical viscosity localized near discontinuities and with
stable time integration under the CFL restriction. In all comparisons we
use the same spatial grid as for PITT, take CFL number $0.45$, reconstruct
the source term at cell centers, and impose no boundary condition at the
horizon side since the characteristic speed vanishes there.}

\subsubsection{Role in the comparisons}

The finite-volume baseline used above already achieves well-balanced
behavior to machine precision at interior grid points and at the right
(far-field) boundary, since the MUSCL reconstruction and the source-term
discretization share the same geometric coefficient $g$ at each interface.
Near the horizon boundary, this property is conditional on the boundary
treatment, reflecting the coordinate singularity of the Schwarzschild
metric at $r=2M$ rather than a limitation of the discretization. A sharper
numerical comparison would still benefit from schemes explicitly designed
and validated for Schwarzschild--Burgers steady states across the full
domain, including near the horizon
~\cite{Bouchut2004,lefloch2021classwellbalancedalgorithmsrelativistic}, as
well as the geometry-preserving and convergence-oriented methods developed
for curved backgrounds more broadly
~\cite{CeylanLeFlochOkutmustur2018,LeFlochMakhlof2014,DongLeFloch2019,
GiesselmannLeFloch2020}.


\subsection{Implementation details}
\label{section=4--4}

\subsubsection{Architecture hyperparameters}
{The Token Transformer uses $N_{\mathrm{layers}}=2$ encoder layers, $8$
attention heads, and feedforward dimension $512$, with
$d_{\mathrm{model}}=128$. The FNO encoder retains $\kappa=32$ Fourier
modes with hidden width $64$. The number of update blocks is
$N_{\mathrm{blocks}}=4$. The physics correction magnitude is
$\epsilon=0.002$, the gate parameters are $\lambda=500$ and $\delta=0.015$,
and the field residual bound is $\mu=0.05$. The smooth shock transition
width follows a scheduled annealing: $w=0.01$ for the first $30\%$ of
training epochs, $w=0.005$ for epochs $30\%$--$70\%$, and $w=0.003$
thereafter, progressively sharpening the reconstructed shock front. The
spatial grid has $N_r=500$ points over $r\in[r_{\min},\,15.0]$, and each
trajectory spans $T=100$ output steps with $\Delta t=0.01$.
All linear layers use Xavier initialization, biases are initialized to
zero, and the activation function in the MLP blocks is GELU. The training
loss is averaged first over the grid and then over the rollout times, so
that changing $N_r$ or $T$ does not change the relative weights in
\eqref{eq:loss}.}

\subsubsection{Loss function}
The total training objective combines three terms:
\be
\label{eq:loss}
\mathcal{L}
  = \mathcal{L}_{L^1}
  + \lambda_{\mathrm{ss}}\,\mathcal{L}_{\mathrm{steady}}
  + \lambda_{\mathrm{RH}}\,\mathcal{L}_{\mathrm{RH}},
\ee
with $\lambda_{\mathrm{ss}}=0.01$ and $\lambda_{\mathrm{RH}}=0.001$.
The primary term is the mean absolute error over all rollout steps:
\be
\label{eq:loss-L1}
\mathcal{L}_{L^1}
  = \frac{1}{T}\sum_{t=1}^{T}
    \bigl\|u^{(t)}_{\mathrm{pred}} - u^{(t)}_{\mathrm{target}}\bigr\|_1.
\ee
{The steady-family regularizer penalizes deviations of the
predicted field from the signed steady-state family
$\{u_{K,\varsigma}\}$ by measuring the spatial variance of the inferred
invariant $K(r)=(1-u^2)/g(r)$ at the final rollout step:}
\be
\label{eq:loss-steady}
\mathcal{L}_{\mathrm{steady}}
  = \mathrm{Var}_r\!\left[\frac{1-(u^{(T)})^2}{g(r)}\right].
\ee
The RH regularizer penalizes discrepancies between the predicted shock
velocity and the analytically computed RH speed:
\be
\label{eq:loss-RH}
\mathcal{L}_{\mathrm{RH}}
  = \frac{1}{T-1}\sum_{t=1}^{T-1}
    \left(
      \frac{r_s^{(t+1)} - r_s^{(t)}}{\Delta t} - \sigma_{RH}^{(t)}
    \right)^{\!2}.
\ee
{Both $\mathcal{L}_{\mathrm{steady}}$ and
$\mathcal{L}_{\mathrm{RH}}$ are applied only to samples whose reference
solution contains a shock. They are suppressed inside rarefaction fans, for
which the correct solution is not a two-branch steady profile separated by
a single interface.}

\subsubsection{Optimization}
The model is trained with AdamW~\cite{Loshchilov2019} at learning rate
$2\times10^{-4}$ and weight decay $0$. A ReduceLROnPlateau scheduler
halves the learning rate when the validation loss fails to improve for
$5$ consecutive epochs. Gradients are clipped to a maximum $\ell^2$ norm
of $0.1$. Training runs for at most $150$ epochs with early stopping
(patience $30$). On $300$ analytic training trajectories with batch size
$4$, the model converges after approximately $143$ epochs, reaching a best
validation loss of $6.20\times10^{-4}$.


\subsection{Evaluation metrics}
\label{section=4--5}

\subsubsection{Field error}

We evaluate the model first by the \emph{field MAE}, the mean
absolute error between the predicted and reference fields, averaged over
time steps and spatial points:
\be
\label{eq:field-mae}
\mathcal{E}_{\mathrm{field}}
  = \frac{1}{T\,N_r}\sum_{t=1}^{T}
    \bigl\|\hat{u}^{(t)} - u^{(t)}\bigr\|_1.
\ee
\subsubsection{Interface errors}

The second diagnostic measures the geometric error in the predicted wave
location. For moving shocks, the \emph{shock position error} is the mean
absolute error in the predicted shock trajectory:
\be
\label{eq:shock-position-error}
\mathcal{E}_{\mathrm{shock}}
  = \frac{1}{T}\sum_{t=1}^{T}
    \bigl|\hat{r}_s^{(t)} - r_s^{(t)}\bigr|,
\ee
{where $r_s^{(t)}$ is obtained from the analytic solution, or
detected as the point of maximum spatial gradient for numerical test cases.
For rarefaction samples, an analogous fan-boundary error can be monitored
by comparing the predicted left and right fan edges with the characteristic
boundaries of the exact Riemann solution; unless otherwise stated, the
tables below report only the field MAE and the shock-position MAE.}

\subsubsection{Structural diagnostics}

In addition to the errors reported in the tables, we monitor three
structure-preservation diagnostics. They measure the violation of the
subluminal bound, the drift of the steady-state invariant, and the defect
in the conservative density \(u/g(r)^2\). For a predicted trajectory
\(\hat u\), we use
\be
\label{eq:structural-diagnostics}
\begin{aligned}
\mathcal{D}_{\mathrm{sub}}
  &:= \max_{t,i}\bigl(\,|\hat u_i^{(t)}|-1\,\bigr)_+,
  \qquad 
  &&
\mathcal{D}_{K}
  &:= \frac{1}{T}\sum_{t=1}^{T}
      \mathrm{Var}_i\!\left[
      \frac{1-(\hat u_i^{(t)})^2}{g(r_i)}
      \right],\\
\mathcal{D}_{m}
  &:= \frac{1}{T}\sum_{t=1}^{T}
      \left|
      \sum_i \frac{\hat u_i^{(t)}-\hat u_i^{(0)}}{g(r_i)^2}\,\Delta r
      \right|.
\end{aligned}
\ee
The first quantity must remain close to zero for physically admissible
velocities. The second is evaluated on regions expected to remain on a
single steady-state branch: for shock samples, we exclude a
$\pm5$-grid-point band around the predicted interface, consistent with
the merging criterion used in \autoref{section=6--1}, since a single
grid point crossing $u\approx0$ inside the transition band causes
 $K\approx1/g(r)$ to become large and would otherwise dominate the variance
without physical significance. The third is intended for finite windows
where boundary fluxes are negligible or prescribed.

We find $\mathcal{D}_{\mathrm{sub}}=0$ on all test trajectories: the
predicted velocity never exceeds the speed of light, so the subluminal
constraint of Corollary~\ref{cor:bounded} is enforced exactly. After
excluding the transition band, $\mathcal{D}_K = 2.20\times10^{-5}\pm
1.50\times10^{-5}$ for steady shocks and $5.75\times10^{-5}\pm
7.17\times10^{-5}$ for moving shocks, comparable in magnitude across both
regimes; this confirms that PITT maintains the steady-state invariant with
small drift away from the shock transition. This behavior is
architectural rather than exact: the physics prior $u_L(r), u_R(r)$ in
\eqref{field-reconstruction} lies on the steady-state manifold by
construction, and the residual correction --- the only term without an
analytical constraint --- is only softly penalized by
$\mathcal{L}_{\mathrm{steady}}$ during training; the small measured
values of $\mathcal{D}_K$ indicate that this soft constraint is effective
in practice. $\mathcal{D}_K$ is not applicable to rarefaction samples,
since the solution does not remain on a single steady branch inside the
fan.

The unmodified $\mathcal{D}_m$ is dominated by the same type of artifact
as the unmodified $\mathcal{D}_K$: the $1/g(r)^2$ weight in the
conservative density diverges as $r\to2M$, so a handful of grid points
closest to the horizon --- where $g(r)$ can fall below $10^{-3}$ --- can
amplify small prediction residuals by several orders of magnitude and
dominate the mass integral. Excluding a sufficiently wide band near
$r_{\min}$, analogous to the transition-band exclusion used for
$\mathcal{D}_K$, reduces $\mathcal{D}_m$ to $O(10^{-4})$ for all three
solution types, comparable in magnitude to $\mathcal{D}_K$. A
$20$-grid-point band ($4\%$ of the domain) suffices for moving shocks and
rarefactions; steady shocks, whose interfaces in our test set lie closer
to the horizon, require an $80$-grid-point band ($16\%$ of the domain) to
clear the shock transition itself. In every case the reduction is a
sharp, order-of-magnitude drop as the exclusion band widens, rather than
convergence to a nonzero plateau, which indicates a geometric
amplification artifact rather than a genuine violation of mass
conservation. These diagnostics do not replace the field error, but they
separate small pointwise errors from possible violations of the
geometry-preserving structure.


\section{Numerical validation}
\label{section=5}

\subsection{Comparison with the numerical baseline}
\label{section=5--1}

Table~\ref{tab:results} reports $\mathcal{E}_{\mathrm{field}}$ on the test
set, separated by solution type. Figure~\ref{fig:trajectories} shows
representative predicted trajectories for each regime, and
Figure~\ref{fig:benchmark} summarizes the benchmark comparison.
{The emphasis on wave-front location, entropy-compatible
propagation, and conservative or Riemann-informed structure is shared by
recent learning approaches for hyperbolic problems, including PINN-based
methods~\cite{Raissi2019,Mao2020,Liu2023,FerrerSanchez2024},
weak or entropy formulations~\cite{DeRyck2024,Oubarka2026},
Godunov- and Riemann-informed networks~\cite{Cassia2025,Patsatzis2024,
Peyvan2024,Wang2025}, neural flux or neural-operator methods
\cite{Kim2025,Liu2024,Liu2025,Lorin2024}, and time-continuous fluid
reconstruction~\cite{Pronk2025}.}

\begin{table}[ht]
\centering
\caption{Test set field MAE by solution type (mean $\pm$ std over $30$
test trajectories, lower is better).}
\begin{tabular}{lcccc}
\hline
Method & Moving & Steady & Rarefaction & Overall \\ \hline
PITT (full)
  & $1.11\text{e-}4$ & $1.97\text{e-}3$ & $9.14\text{e-}5$ & $7.25\text{e-}4$ \\
  & ${\scriptstyle\pm4.8\text{e-}5}$ & ${\scriptstyle\pm1.9\text{e-}4}$
  & ${\scriptstyle\pm3.5\text{e-}5}$ & ${\scriptstyle\pm8.9\text{e-}4}$ \\
Numerical (RK3+HLL)
  & $2.80\text{e-}3$ & $2.02\text{e-}3$ & $1.41\text{e-}3$ & $2.08\text{e-}3$ \\
  & ${\scriptstyle\pm7.9\text{e-}4}$ & ${\scriptstyle\pm6.2\text{e-}4}$
  & ${\scriptstyle\pm4.3\text{e-}4}$ & ${\scriptstyle\pm8.5\text{e-}4}$ \\
\hline
\end{tabular}
\label{tab:results}
\end{table}

\subsubsection{Overall accuracy}
{PITT gives an overall MAE of $7.25\times10^{-4}$, while the numerical
baseline gives $2.08\times10^{-3}$. The gain is not uniform across
solution types, and its structure reflects the role of the RH prior
established in Proposition~\ref{prop:closed}.}

\subsubsection{Moving shocks}
The most pronounced improvement occurs for moving shocks, where PITT
reduces the MAE by a factor of $25$ ($1.11\times10^{-4}$ vs
$2.80\times10^{-3}$). By evaluating the exact RH speed $\sigma_{RH}^{(t)}$
analytically at each rollout step via \eqref{eq:shock-ode}, PITT tracks
the shock trajectory with high precision, while the numerical scheme
accumulates positional error due to numerical diffusion across the
discontinuity.

\subsubsection{Steady shocks}
For steady shocks ($\sigma_{RH}=0$), the two methods perform comparably
($1.97\times10^{-3}$ vs $2.02\times10^{-3}$). Since the shock does not
move, the numerical solver does not need to track a trajectory, and
numerical diffusion near the stationary discontinuity affects both methods
similarly. The gate $\gamma$ suppresses the neural correction in this
regime, so PITT's output reduces to the physics prior evaluated at the
fixed shock position.

\subsubsection{Rarefaction waves}
{PITT gives a smaller error than the numerical solver on rarefaction waves
($9.14\times10^{-5}$ vs $1.41\times10^{-3}$), despite the piecewise
steady-state prior being an approximation in this case. The learned
residual correction accounts for the smooth fan structure, compensating for
the mismatch between the sharp-transition prior and the continuous
rarefaction profile.}

\begin{figure}[H]
\centering
\includegraphics[width=0.32\textwidth]{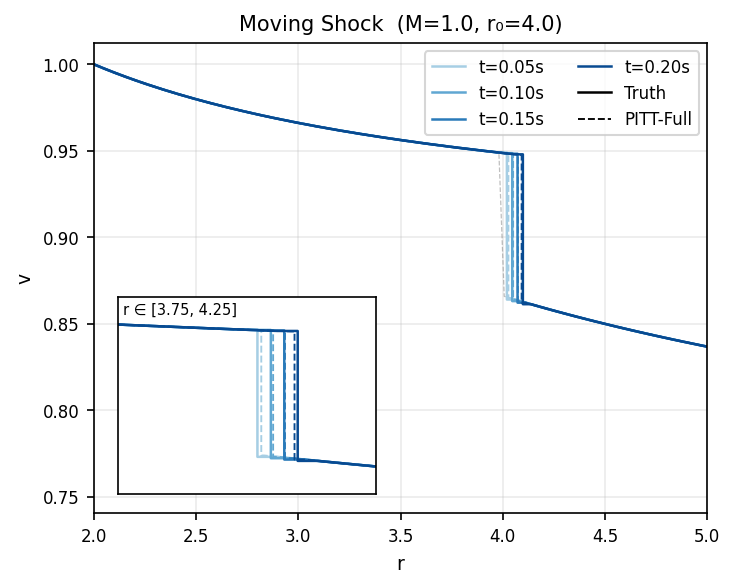}
\includegraphics[width=0.32\textwidth]{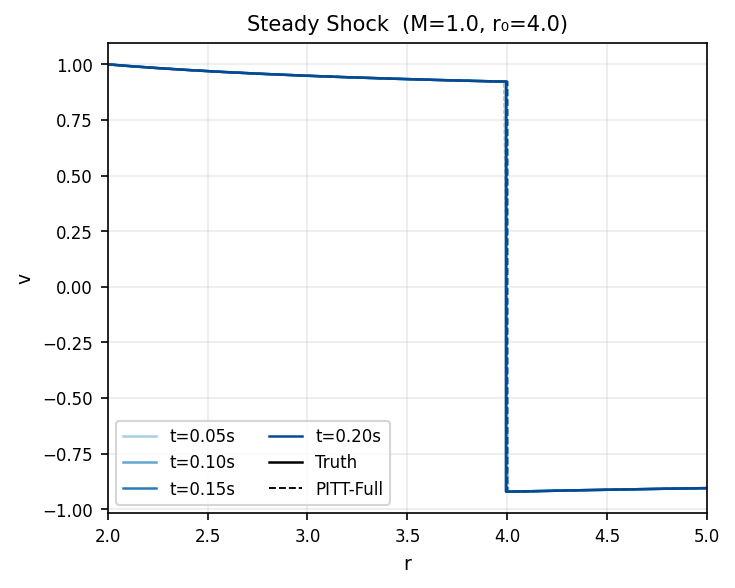}
\includegraphics[width=0.32\textwidth]{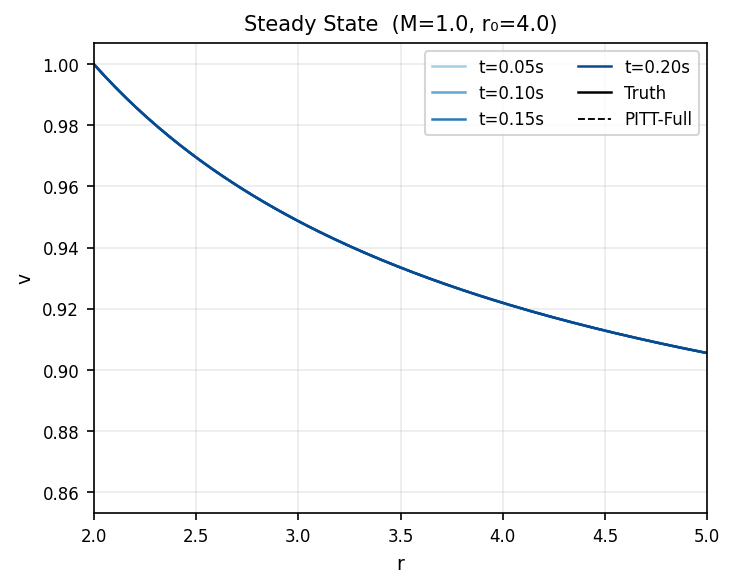}
\caption{Predicted velocity profiles at $t\in\{0.05,0.10,0.15,0.20\}$s
for three representative test cases ($M=1.0$, $r_0=4.0$). Solid lines:
exact solution; dashed lines: PITT-Full. Left: moving shock
($K_L=0.2$, $K_R=0.5$, $\varsigma_L=\varsigma_R=+1$); center: steady
shock ($K_L=K_R=0.3$, $\varsigma_L=-\varsigma_R=+1$); right: steady
state ($K_L=K_R=0.3$, $\varsigma_L=\varsigma_R=+1$). In all cases the
two curves are visually indistinguishable.}
\label{fig:trajectories}
\end{figure}

\begin{figure}[H]
\centering
\includegraphics[width=0.95\textwidth]{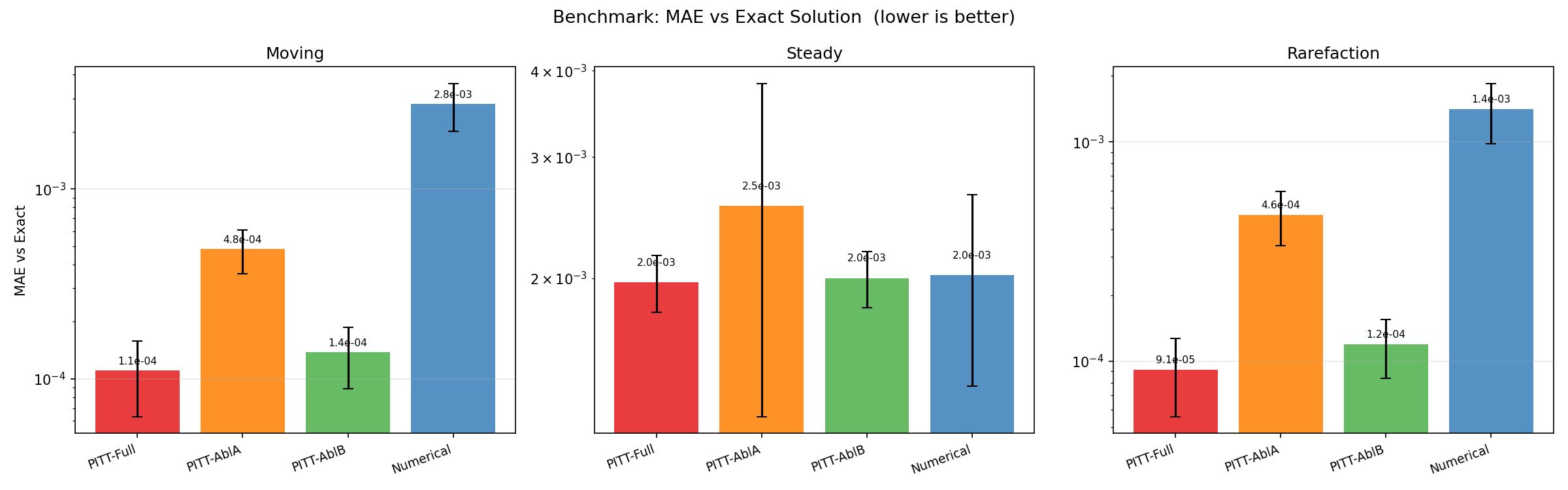}
\caption{Benchmark summary for the test set. The figure compares the
PITT prediction with the finite-volume baseline across the regimes used in
Table~\ref{tab:results}.}
\label{fig:benchmark}
\end{figure}


\subsection{Ablation study}
\label{section=5--2}

We compare the full model against the following three variants. \textbf{PITT-AblA}
removes the RH prior from the forward pass, replacing $\sigma_{RH}^{(t)}$
with a purely learned shock speed. \textbf{PITT-AblB} disables the
equation tokenizer by zeroing $\mathbf{e}_{\mathrm{phys}}$ and the token
cross-attention context. \textbf{PITT-AblC} removes the shock-width
annealing schedule, training instead with a fixed $w=0.003$ throughout.
Because some of these variants turn out to be sensitive to the random
initialization, we report results directly across three seeds ($2026$,
$1$, $42$) rather than for a single run. Table~\ref{tab:ablation-seeds}
reports the overall field MAE for every variant and every seed,
Table~\ref{tab:ablation-crossseed} summarizes the cross-seed mean and
standard deviation, and Table~\ref{tab:ablation-seeds-bytype} breaks
these results down by solution type.

\begin{table}[ht]
\centering
\caption{Overall field MAE for each variant under three random seeds.}
\begin{tabular}{lccc}
\hline
Model & seed $2026$ & seed $1$ & seed $42$ \\ \hline
PITT (full) & $2.88\text{e-}4$ & $2.88\text{e-}4$ & $3.12\text{e-}4$ \\
PITT-AblA & $6.24\text{e-}4$ & $6.99\text{e-}4$ & $2.91\text{e-}4$ \\
PITT-AblB & $3.17\text{e-}4$ & $3.07\text{e-}4$ & $2.99\text{e-}4$ \\
PITT-AblC & $2.85\text{e-}4$ & $1.75\text{e-}3$ & $2.91\text{e-}4$ \\
\hline
\end{tabular}
\label{tab:ablation-seeds}
\end{table}

\begin{table}[ht]
\centering
\caption{Cross-seed variation of the overall field MAE (mean $\pm$ std
across three seeds). AblC and AblA show large relative variation; AblB
is nearly insensitive to the seed over these three runs.}
\begin{tabular}{lccc}
\hline
Model & Cross-seed mean $\pm$ std & Relative variation & Worst seed / mean \\ \hline
PITT (full) & $2.96\text{e-}4 \pm 1.14\text{e-}5$ & $4\%$ & $1.05\times$ \\
PITT-AblA & $5.38\text{e-}4 \pm 1.77\text{e-}4$ & $33\%$ & $1.30\times$ \\
PITT-AblB & $3.08\text{e-}4 \pm 7.40\text{e-}6$ & $2\%$ & $1.03\times$ \\
PITT-AblC & $7.76\text{e-}4 \pm 6.91\text{e-}4$ & $89\%$ & $2.26\times$ \\
\hline
\end{tabular}
\label{tab:ablation-crossseed}
\end{table}

\begin{table}[ht]
\centering
\caption{Field MAE by solution type for each variant, across three
random seeds (mean $\pm$ std over test trajectories).}
\begin{tabular}{llccc}
\hline
Model & Seed & Moving & Steady & Rarefaction \\ \hline
PITT (full) & $2026$ & $1.11\text{e-}4 \pm 1.22\text{e-}4$ & $8.28\text{e-}4 \pm 3.61\text{e-}4$ & $5.63\text{e-}5 \pm 3.71\text{e-}5$ \\
PITT (full) & $1$    & $1.11\text{e-}4 \pm 1.22\text{e-}4$ & $8.28\text{e-}4 \pm 3.62\text{e-}4$ & $5.62\text{e-}5 \pm 3.75\text{e-}5$ \\
PITT (full) & $42$   & $1.39\text{e-}4 \pm 1.24\text{e-}4$ & $8.52\text{e-}4 \pm 3.61\text{e-}4$ & $7.48\text{e-}5 \pm 3.78\text{e-}5$ \\
\hline
PITT-AblA & $2026$ & $4.39\text{e-}4 \pm 2.09\text{e-}4$ & $1.13\text{e-}3 \pm 3.66\text{e-}4$ & $4.41\text{e-}4 \pm 1.13\text{e-}4$ \\
PITT-AblA & $1$    & $5.71\text{e-}4 \pm 1.59\text{e-}4$ & $1.21\text{e-}3 \pm 3.78\text{e-}4$ & $4.06\text{e-}4 \pm 8.47\text{e-}5$ \\
PITT-AblA & $42$   & $1.16\text{e-}4 \pm 1.22\text{e-}4$ & $8.30\text{e-}4 \pm 3.61\text{e-}4$ & $5.86\text{e-}5 \pm 3.75\text{e-}5$ \\
\hline
PITT-AblB & $2026$ & $1.41\text{e-}4 \pm 1.22\text{e-}4$ & $8.58\text{e-}4 \pm 3.60\text{e-}4$ & $8.57\text{e-}5 \pm 3.69\text{e-}5$ \\
PITT-AblB & $1$    & $1.31\text{e-}4 \pm 1.22\text{e-}4$ & $8.46\text{e-}4 \pm 3.62\text{e-}4$ & $7.55\text{e-}5 \pm 3.76\text{e-}5$ \\
PITT-AblB & $42$   & $1.23\text{e-}4 \pm 1.22\text{e-}4$ & $8.40\text{e-}4 \pm 3.61\text{e-}4$ & $6.80\text{e-}5 \pm 3.70\text{e-}5$ \\
\hline
PITT-AblC & $2026$ & $1.08\text{e-}4 \pm 1.22\text{e-}4$ & $8.25\text{e-}4 \pm 3.61\text{e-}4$ & $5.38\text{e-}5 \pm 3.76\text{e-}5$ \\
PITT-AblC & $1$    & $1.51\text{e-}3 \pm 1.08\text{e-}3$ & $2.28\text{e-}3 \pm 1.19\text{e-}3$ & $1.66\text{e-}3 \pm 1.05\text{e-}3$ \\
PITT-AblC & $42$   & $1.14\text{e-}4 \pm 1.23\text{e-}4$ & $8.30\text{e-}4 \pm 3.61\text{e-}4$ & $5.92\text{e-}5 \pm 3.75\text{e-}5$ \\
\hline
\end{tabular}
\label{tab:ablation-seeds-bytype}
\end{table}

\subsubsection{Effect of the RH prior}
Removing the RH prior causes a large degradation at two
of the three seeds: the overall MAE increases by $117\%$ and $143\%$
relative to PITT-full at seeds $2026$ and $1$ respectively, with the
sharpest relative impact on rarefaction waves and moving shocks. At the
third seed ($42$), AblA performs close to PITT-full. Without the
analytic prior provided by Proposition~\ref{prop:closed}, the model must
infer the shock speed from training data, which is substantially harder
given only $300$ training trajectories, and the resulting optimization
problem is more sensitive to initialization.

\subsubsection{Effect of the equation tokenizer}
Removing the tokenizer produces a smaller but consistent degradation
across all three seeds: the overall MAE increases by $7$--$10\%$, with
the largest relative impact on rarefaction waves and moving shocks. The
FNO encoder and the analytic RH prior together provide a strong
inductive bias that partially compensates for the absence of symbolic
encoding, but the full model remains more accurate at every seed,
confirming that symbolic equation tokenization contributes a measurable,
if secondary and seed-independent, signal.

\subsubsection{Effect of the shock-width annealing schedule}
At two of the three seeds, PITT-AblC (fixed $w=0.003$ throughout
training) performs essentially as well as PITT-full, or even slightly
better. At the remaining seed, however, training converges to a
solution that is worse than PITT-full by roughly an order of magnitude
across \emph{every} solution type simultaneously -- moving shocks,
steady shocks, and rarefaction waves all degrade together
(Table~\ref{tab:ablation-seeds-bytype}), rather than one wave regime
being selectively affected. This across-the-board collapse indicates a
genuine optimization failure rather than a regime-specific weakness, and
parallels the behavior of PITT-AblA: the annealing schedule is not
primarily an accuracy-improving component, but a mechanism that removes
a failure mode in which the (initially very sharp) transition-width
optimization problem can converge to a poor local optimum depending on
initialization.

\subsubsection{Summary}
Taken together, these results separate the three architectural
components into two groups. The RH prior (AblA) and the shock-width
annealing schedule (AblC) primarily affect training robustness: removing
either one leaves the optimization problem vulnerable, under some
initializations, to a substantially worse local optimum, while under
other initializations the variant performs close to the full model. The
equation tokenizer (AblB) instead affects accuracy in a consistent,
seed-independent way, contributing a modest but reliable gain. A
single-seed ablation study would not be sufficient to assess the role of the RH prior and could miss the annealing schedule.

A more complete ablation program could also compare the present
architecture with variants in which the FNO encoder, the residual field
correction, or the steady-state variable $K$ are removed separately, and
could extend the cross-seed analysis above to a larger number of seeds
to better characterize the failure probability associated with removing
the RH prior or the annealing schedule.


\subsection{Non-relativistic limit: standard Burgers equation}
\label{section=5--3}

\subsubsection{Flat-spacetime reduction}

Taking the limit as the black-hole mass $M\to0$, equation
\eqref{burgers-schwarzschild} reduces to the standard inviscid Burgers
equation
\be
\label{eq:standard-burgers}
\partial_t u + \partial_r\!\left(\frac{u^2}{2}\right) = 0. 
\ee
The Rankine--Hugoniot condition gives the shock speed
$\sigma_{RH}=(u_L+u_R)/2$, and the steady states reduce to the constants
$u_{K,\varsigma}=\varsigma\sqrt{1-K}$. This limit is recovered exactly in
the PITT prior: when $g(r)=1-2M/r\to1$, the flux reduces to
$F=(u^2-1)/2$ and the RH condition gives
\be
\label{eq:standard-burgers-rh-speed}
\sigma_{RH} = \frac{[F]}{[u]}
= \frac{u_R^2 - u_L^2}{2(u_R - u_L)}
= \frac{u_R + u_L}{2},
\ee
identically.

\subsubsection{Numerical check}

We verify this numerically by running PITT for
$M\in\{10^{-4},10^{-5},10^{-6}\}$. Table~\ref{tab:standard} reports the
full results. The field MAE remains essentially constant across these
values of $M$, while the shock-speed discrepancy decreases linearly with
$M$, as expected from the expansion of the Schwarzschild metric coefficient
$g(r)$.

\begin{table}[ht]
\centering
\caption{Standard Burgers validation ($M=10^{-6}$, $T=0.5$).
PITT vs RK3+MUSCL+HLL numerical baseline.}
\begin{tabular}{lcccc}
\hline
Case & $\sigma_{RH}$ & $|\Delta\sigma|$ & PITT MAE & Numerical MAE \\ \hline
Moving shock ($u_L=0.8$, $u_R=0.3$)
  & $0.550$ & $1.31\times10^{-7}$ & $3.5\times10^{-3}$ & $4.8\times10^{-3}$ \\
Rarefaction ($u_L=0.3$, $u_R=0.8$)
  & --- & --- & $1.9\times10^{-3}$ & $1.16\times10^{-2}$ \\
Steady shock ($u_L=0.5$, $u_R=-0.5$)
  & $0$ & $0$ & $6.9\times10^{-3}$ & $9.5\times10^{-3}$ \\ \hline
\end{tabular}
\label{tab:standard}
\end{table}
For the moving shock case, the predicted shock speed exhibits
empirical first-order convergence to the standard Burgers RH value as
$M\to0$:
\be
\label{eq:standard-burgers-speed-convergence}
|\Delta\sigma| \approx 1.3\times10^{-1} \cdot M,
\ee
as confirmed by the measurements
\be
\label{eq:standard-burgers-speed-measurements}
\begin{aligned}
M=10^{-4}&:\;|\Delta\sigma|=1.28\times10^{-5},\\
M=10^{-5}&:\;|\Delta\sigma|=1.26\times10^{-6},\\
M=10^{-6}&:\;|\Delta\sigma|=1.31\times10^{-7}.
\end{aligned}
\ee


\section{Long-time convergence toward steady states}
\label{section=6}

\subsection{Piecewise steady-state solutions}
\label{section=6--1}
 
\subsubsection{Algorithm for general solutions}

{Given the initial condition $u_0(r)$, we proceed as follows.}
\begin{enumerate}
\item {Compute $K(r)=(1-u_0^2)/g(r)$ pointwise.}
\item Detect discontinuities of $K$ by thresholding
      $|\partial_r K| > \tau$; detections within $5$ grid points
      are merged to avoid double-counting.
\item For each segment $[r_{i-1}, r_i)$, set $K^i$ to the segment
      mean of $K(r)$ and $\varsigma^i$ to the sign of the segment
      mean of $u_0$.
\item {For each adjacent pair
      $(K^i,\varsigma^i)$, $(K^{i+1},\varsigma^{i+1})$, solve the
      corresponding generalized Riemann problem: shock interfaces are
      advanced by \eqref{eq:shock-ode}, while rarefaction interfaces are
      advanced by the associated fan characteristics.}
\item {Reconstruct the full field by piecewise evaluation of
      $u_{K^i,\varsigma^i}(r)$, with sigmoid transitions of width $w$ at
      shock locations and smooth fan reconstruction in rarefaction
      regions.}
\end{enumerate}
{
This procedure follows the spirit of the Glimm method based on generalized
Riemann problems~\cite{LeFloch2019, LeFloch2016}: at each time step
$\Delta t$, the field is decomposed into piecewise steady-state segments
via the invariant $K(r)$ and the sign branch $\varsigma$. Each interface
generates a local Riemann subproblem, and the local solutions are
reassembled before the next step. To keep the pre-interaction regime
unambiguous within one time step, we impose the CFL condition}
\be
\label{eq:cfl}
\frac{\Delta t}{\Delta r} \leq \frac{1}{\displaystyle\sup_{r\in[r_{\min},r_{\max}]}g(r)}
= \frac{1}{g(r_{\max})},
\ee
{which is sufficient here since characteristic speeds are bounded
by $g(r)|u|\leq g(r)$ and $|u|\leq1$ by Corollary~\ref{cor:bounded}.
Under this condition, wave interactions are handled by re-decomposition at
successive time steps, in agreement with the front-tracking and generalized
Riemann-solver theory developed by LeFloch and
Xiang~\cite{LeFloch2016}.}

\subsubsection{Numerical results for general solutions}

All experiments in this
section use $T=0.2$ with $\Delta t=0.01$ and $\Delta r =
(r_{\max}-r_{\min})/N_r$, satisfying \eqref{eq:cfl}.

{Table~\ref{tab:general} reports the field MAE and shock position
MAE for three test cases with $M=1.0$, and Figure~\ref{fig:general} shows
the corresponding velocity profiles and shock trajectories.}

\begin{table}[ht]
\centering
\caption{{PITT-General: field MAE and shock position MAE for
general bounded-variation initial data ($M=1.0$, $T=0.2$, pre-collision regime). Shock
MAE is reported per discontinuity.}}
\begin{tabular}{llcc}
\hline
Case & $K$ values & Field MAE & Shock MAE \\ \hline
2-disc
  & $(0.2\,|\,0.5\,|\,0.7)$
  & $1.03\text{e-}4$
  & $3.84\text{e-}3\;/\;4.60\text{e-}3$ \\
3-disc
  & $(0.2\,|\,0.4\,|\,0.6\,|\,0.8)$
  & $1.48\text{e-}4$
  & $3.89\text{e-}3\;/\;4.81\text{e-}3\;/\;4.81\text{e-}3$ \\
Smooth IC
  & $K\approx0.35$ + perturbation $|\delta u|<0.02$
  & $4.58\text{e-}3$
  & --- \\ \hline
\end{tabular}
\label{tab:general}
\end{table}

{
Two observations follow from Table~\ref{tab:general} and
Figure~\ref{fig:general}. First, the field MAE for the 2- and
3-discontinuity cases ($1.03\times10^{-4}$ and $1.48\times10^{-4}$) is
comparable to the single-discontinuity result in Table~\ref{tab:results}
($1.11\times10^{-4}$). Thus, at this resolution and before wave
interaction, the decomposition into local generalized Riemann problems
does not visibly degrade the accuracy of the single-interface solver.
Second, for the smooth initial condition, consisting of a steady-state
profile perturbed by $|\delta u|<0.02$, the field MAE is
$4.58\times10^{-3}=O(|\delta u|)$. This error reflects the initial
projection of a genuinely smooth perturbation onto a reduced steady-state
description, rather than an instability of the PITT rollout.}

These tests for rather general solutions should be interpreted as interaction tests of the
local generalized Riemann decomposition. The more stringent regime is
the post-collision evolution, in which two or more fronts interact and
the steady-state parameters must be reselected after the interaction.
This regime is numerically accessible with the same decomposition
strategy, but it requires a separate error study because the dominant
error is then no longer only the single-front propagation error. We
regard this post-interaction regime as the next necessary benchmark for
extending the method from isolated generalized Riemann problems to fully general flows, as noted among the limitations discussed in
\autoref{section=8}.

\begin{figure}[H]
\centering
\includegraphics[width=\textwidth]{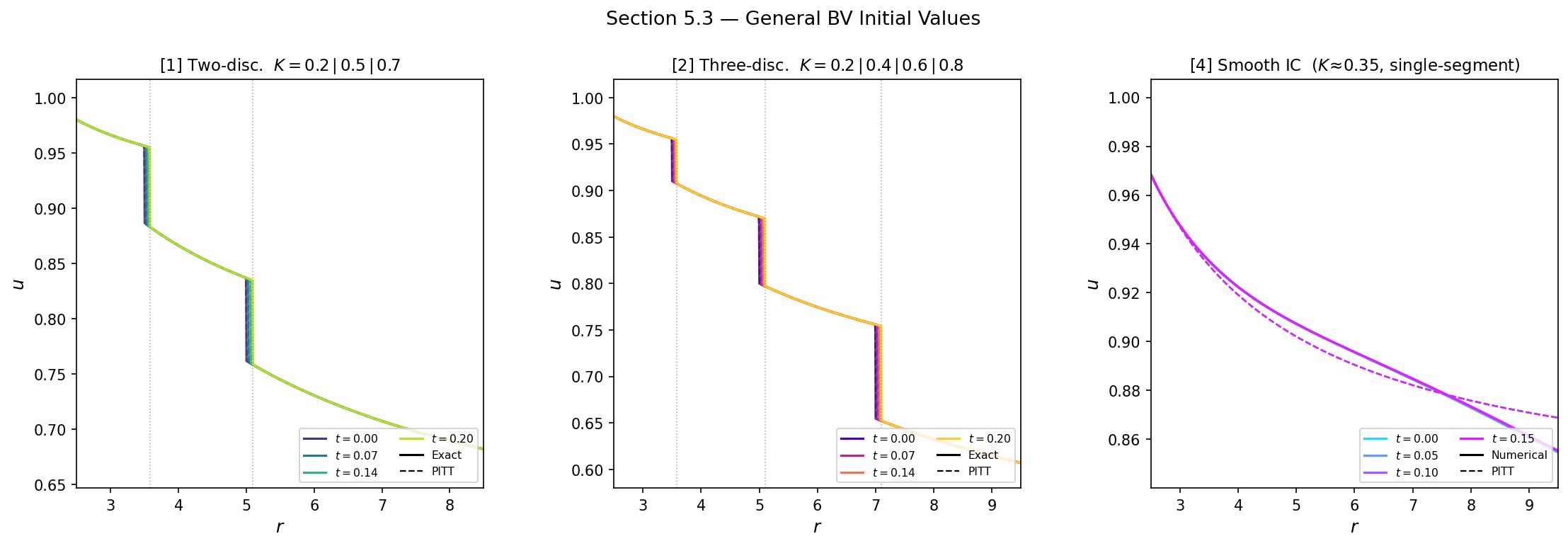}
\caption{{PITT-General on general bounded-variation initial data ($M=1.0$).
\textbf{Left:} two-discontinuity case ($K=0.2\,|\,0.5\,|\,0.7$);
solid lines are the exact solution, dashed lines are PITT.
\textbf{Center:} three-discontinuity case ($K=0.2\,|\,0.4\,|\,0.6\,|\,0.8$);
solid lines are the exact solution, dashed lines are PITT.
In both cases the two curves are visually indistinguishable
(field MAE $\sim10^{-4}$); vertical dotted lines mark the
initial discontinuity positions.
\textbf{Right:} smooth initial condition (steady-state
$K\approx0.35$ with small perturbation $|\delta u|<0.02$);
the field MAE of $4.58\times10^{-3}$ is $O(|\delta u|)$,
attributable to the initial approximation rather than the
PITT rollout.}}
\label{fig:general}
\end{figure} 


\subsection{Asymptotic behavior of solutions}
\label{section=6--2}

\subsubsection{Purpose of the asymptotic tests}

We now study the late-time behavior of solutions to the relativistic
Burgers equation \eqref{burgers-schwarzschild} on a Schwarzschild
background, providing numerical support for the asymptotic regimes
described in \cite{LeFloch2014}. All experiments use $M=1$, $r\in[2,4]$, $N=256$, and the finite-volume
scheme of \autoref{section=4--3}, which is well-balanced to machine
precision away from the horizon. For perturbed steady shocks, the limiting
shock location depends on the mass displaced by the perturbation; this
motivates the separate conservation-based discussion in
\autoref{section=7}.

\subsubsection{Perturbed smooth steady state}

We take the steady state
$u_*(r)=-\sqrt{1-K\,g(r)}$ with $K=0.5$ as initial data, and add a
compactly supported sinusoidal perturbation on $r\in[2.2,2.8]$.
Figure~\ref{fig:asymptotic9} shows the evolution at
$t\in\{0.2,0.5,2,4\}$. The perturbation is gradually absorbed toward the
black-hole horizon, and the solution converges back to the same steady
state.

\subsubsection{Perturbed steady shock}

We take the steady shock with $K=0.5$ and $r_0=2.5$ as initial data, with
compactly supported perturbations added on both sides of the discontinuity.
Figure~\ref{fig:asymptotic10} shows the evolution at
$t\in\{0.1,2,5,10\}$. The perturbation dissipates and the solution
converges to a steady shock with the same $K$ value, possibly at a shifted
location $r_1\neq r_0$, consistent with Conclusion~6.1 of
\cite{LeFloch2014}.

\subsubsection{General data with horizon trace \(u_0(2M)=1\)}

We prescribe a piecewise initial velocity with $u_0(2M)=1$ and
$u_0(+\infty)=1$. Figure~\ref{fig:asymptotic11} shows the evolution at
$t\in\{0.1,1,5,8,10,40\}$. The solution develops a single shock connecting
the left state $u=1$ to the escape velocity profile
$u^-_*(r)=-\sqrt{2M/r}$, consistent with Claim~1.2(1) of
\cite{LeFloch2014}. The relatively large $L^1$ distance
($5.0\times10^{-1}$ at $T=40$) reflects the contribution of the left
region where $u\approx1$, not a failure of convergence.

{
\subsubsection{General data with positive far-field trace}

We prescribe $u_0(r)=0.5+\delta(r)$ for $r<3$ and $u_0(r)=0.08$ for $r>3$,
where $\delta$ is a small cosine bump on $[2.0,2.6]$.
Figure~\ref{fig:asymptotic12} shows the evolution at
$t\in\{0.1,1,4,10,15,20\}$. Since $u_0(2M)<1$ and
$u_0(+\infty)=0.08>0$, the solution converges to the escape velocity
profile $u^-_*(r)=-\sqrt{2M/r}$, consistent with Claim~1.2(2) of
\cite{LeFloch2014}. At $T=20$ the $L^1$ distance is
$1.57\times10^{-2}$, showing significantly faster convergence than in the
case $u_0(2M)=1$.}

\begin{figure}[H]
\centering
\includegraphics[width=\textwidth]{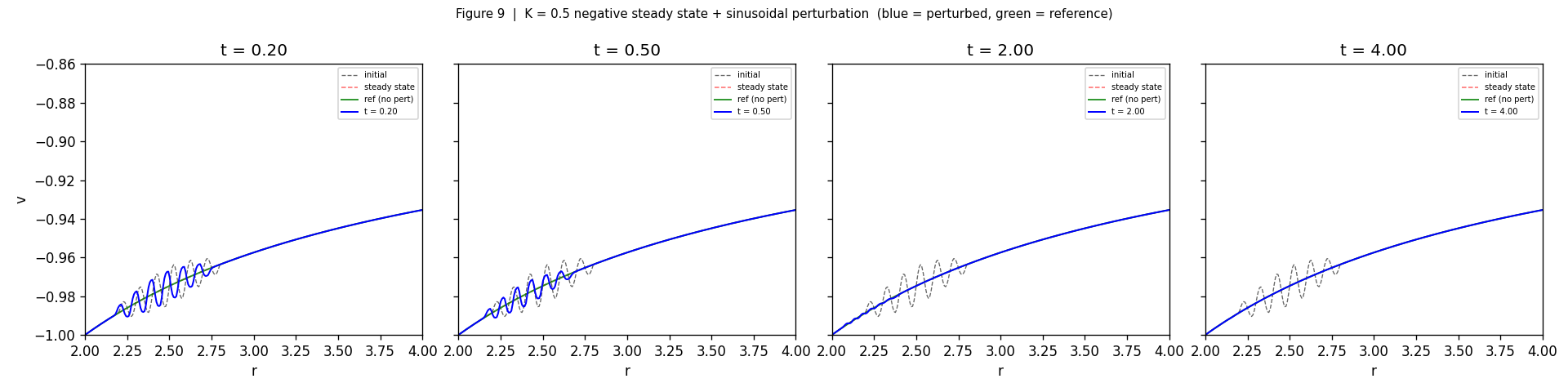}
\caption{Asymptotic behavior: perturbed smooth steady state ($K=0.5$,
$M=1$). Blue: perturbed solution; green: unperturbed reference;
red dashed: analytical steady state. The perturbation is absorbed
toward the horizon and the solution returns to the steady state.}
\label{fig:asymptotic9}
\end{figure}

\begin{figure}[H]
\centering
\includegraphics[width=\textwidth]{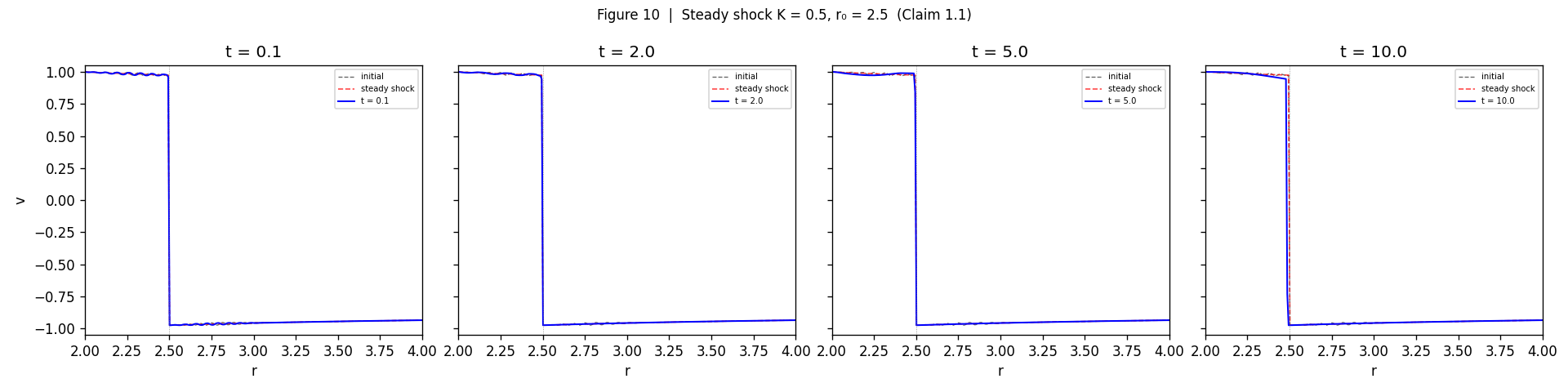}
\caption{Asymptotic behavior: perturbed steady shock ($K=0.5$,
$r_0=2.5$, $M=1$). The perturbation dissipates and the solution
converges to a steady shock at a possibly shifted location
$r_1\neq r_0$.}
\label{fig:asymptotic10}
\end{figure}

\begin{figure}[H]
\centering
\includegraphics[width=\textwidth]{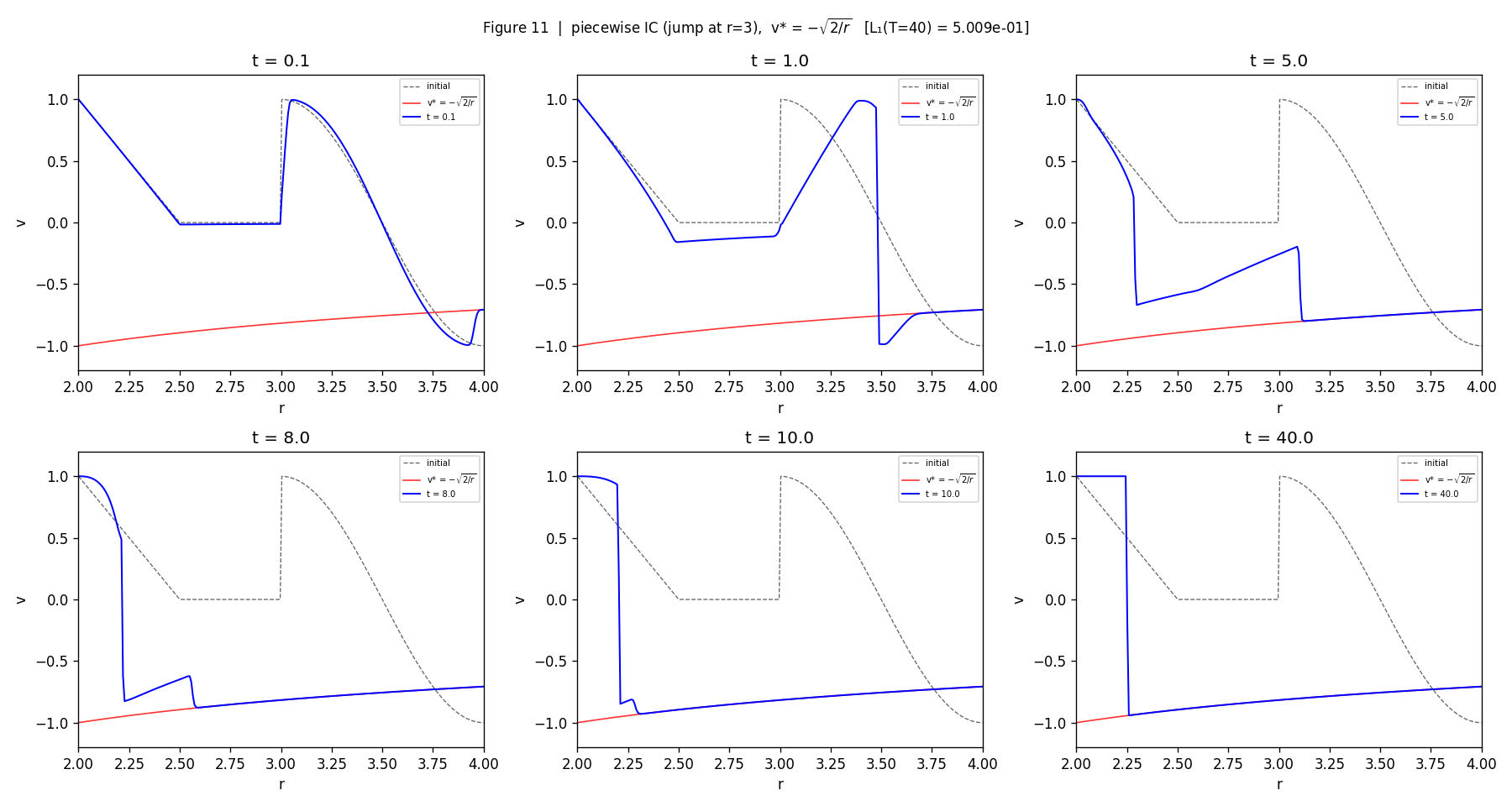}
\caption{{Asymptotic behavior: general initial data with
$u_0(2M)=1$ and $u_0(+\infty)=1$ ($M=1$, $T=40$). The solution converges
to a single shock connecting the left state $u=1$ to the escape velocity
profile $u^-_*(r)=-\sqrt{2M/r}$ (red), consistent with Claim~1.2(1) of
\cite{LeFloch2014}.}}
\label{fig:asymptotic11}
\end{figure}

\begin{figure}[H]
\centering
\includegraphics[width=\textwidth]{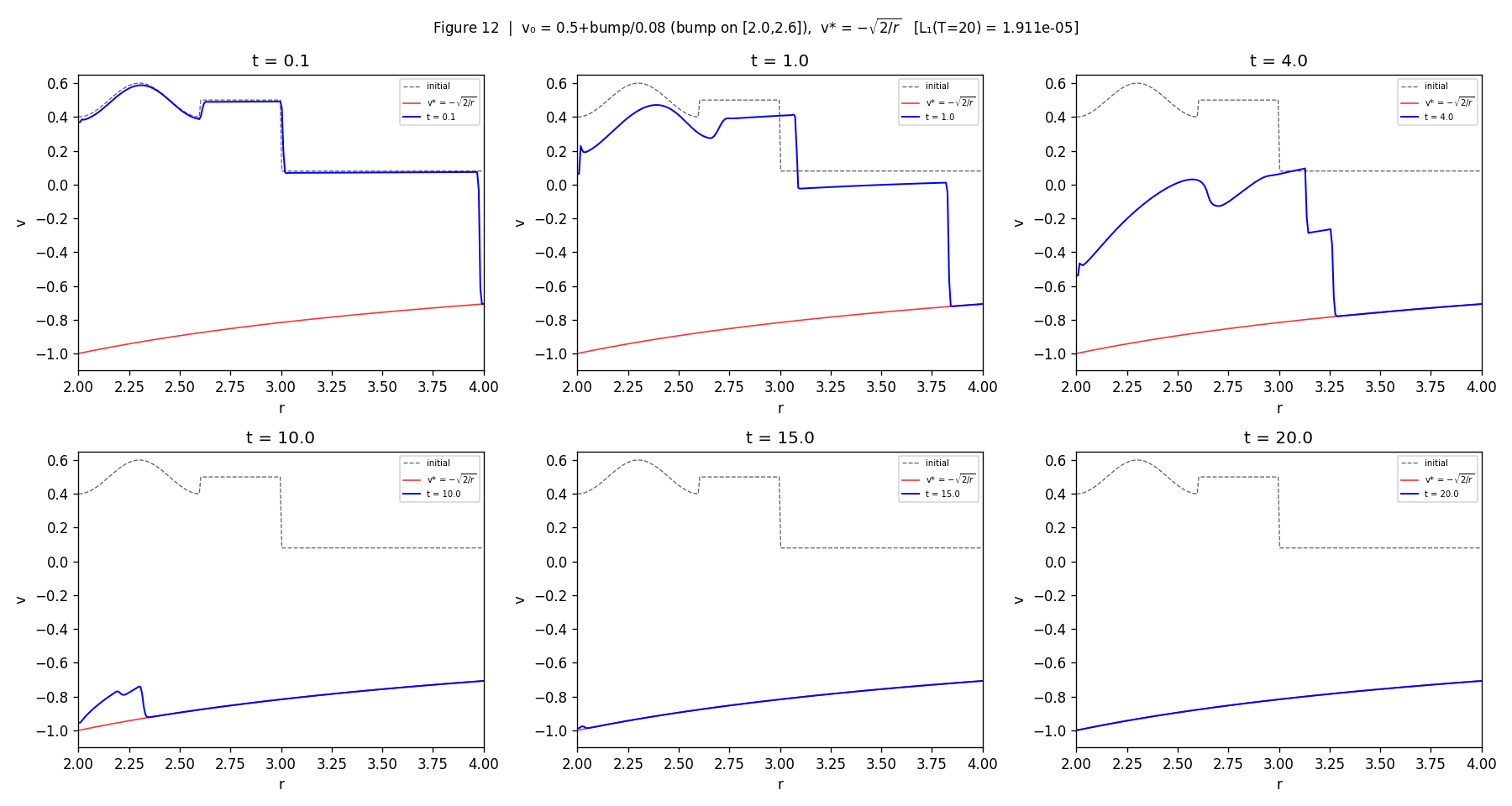}
\caption{{Asymptotic behavior: general initial data with
$u_0(2M)<1$ and $u_0(+\infty)=0.08>0$ ($M=1$, $T=20$). The solution
converges to the escape velocity profile $u^-_*(r)=-\sqrt{2M/r}$ (red),
consistent with Claim~1.2(2) of \cite{LeFloch2014}. $L^1$ distance at
$T=20$: $1.57\times10^{-2}$.}}
\label{fig:asymptotic12}
\end{figure}

{
\subsubsection{General data with nonpositive far-field trace}

We prescribe $u_0(r)=0.6+0.02\cos(2\pi(r-2))$ for $r\in[2,3]$ and
$u_0(r)=-0.6+0.02\cos(3\pi(r-3))$ for $r\in[3,4]$, with
$u_0^\infty=-0.6\leq0$. Figure~\ref{fig:asymptotic13} shows the evolution
at $t\in\{0.1,1,4,8,10,20\}$. The solution converges to}
\be
\label{eq:asymptotic-nonpositive-far-field-state}
u^*(r) = -\sqrt{1-(1-(u_0^\infty)^2)\,g(r)}
= -\sqrt{1-0.64\,g(r)},
\ee
{consistent with Claim~1.2(3) of \cite{LeFloch2014}. The
convergence is rapid: the $L^1$ distance at $T=20$ is
$1.05\times10^{-5}$.}

\begin{figure}[H]
\centering
\includegraphics[width=\textwidth]{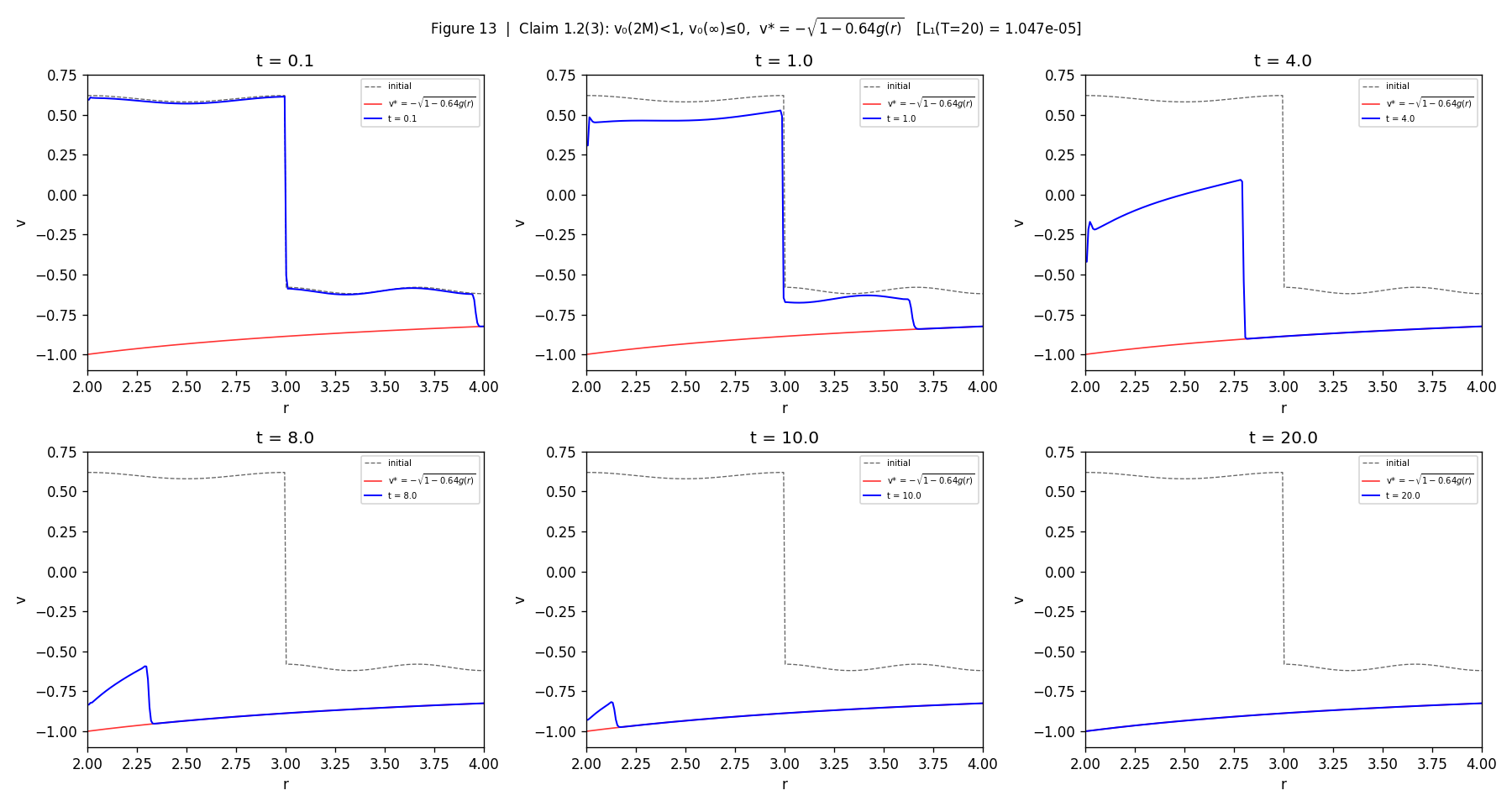}
\caption{{Asymptotic behavior: general initial data with
$u_0(2M)<1$ and $u_0(+\infty)=-0.6\leq0$ ($M=1$, $T=20$). The solution
converges to $u^*(r)=-\sqrt{1-0.64\,g(r)}$ (red), consistent with
Claim~1.2(3) of \cite{LeFloch2014}. $L^1$ distance at
$T=20$: $1.05\times10^{-5}$.}}
\label{fig:asymptotic13}
\end{figure}


\section{Mass conservation and shock displacement for perturbed steady shocks}
\label{section=7}

\subsection{Mass balance for perturbed steady shocks}
\label{section=7--1}

\subsubsection{Question addressed in this section}

We now isolate one asymptotic question suggested by the numerical
experiments above. Consider a steady shock for the Schwarzschild--Burgers
equation, and perturb it by a compactly supported disturbance. If the
solution relaxes as $t\to\infty$ to another steady shock in the same
steady-state family, we seek what determines the limiting shock position. 
LeFloch, Par\'es, and Pimentel~\cite{lefloch2021classwellbalancedalgorithmsrelativistic}
observed numerically that, for small perturbations, the displacement is
approximately linear in the perturbation amplitude. We show here that the
underlying relation follows directly from the conservative form of the
equation and can be used as a sharp diagnostic for PITT predictions.

\subsubsection{Conservative density}

Equation~\eqref{burgers-2} was introduced as a balance law with source
term $S$, but it is equivalently written in the conservative form
\be
\label{eq:conservative-form}
\partial_t\!\left(\frac{u}{g(r)^2}\right)
+ \partial_r\!\left(\frac{u^2-1}{2g(r)}\right) = 0.
\ee
Thus the conserved density is not $u$ itself, but the weighted density
$u/g(r)^2$. On a finite computational interval $[a,b]$, with boundary
fluxes either prescribed or negligible over the relaxation window, the
weighted mass provides the scalar constraint that determines the final
shock location.

\subsubsection{Perturbed steady shock}

To match the notation used in Figure~\ref{fig:displacement}, we write
$\kappa_0=K_0^2$ for the steady-state parameter in this subsection and set
\(U_{\kappa_0}(r):=\sqrt{1-\kappa_0 g(r)}\).
The reference steady shock is
\be
\label{eq:steady-shock-ic}
u_0(r) = \begin{cases}
U_{\kappa_0}(r), & 2M<r<r_0,\\
-U_{\kappa_0}(r), & r>r_0.
\end{cases}
\ee
We perturb $u_0$ by a compactly supported disturbance $\delta(r)$, and
write the initial field as $u_0+\delta$. If the perturbed solution
converges to a steady shock in the same $\kappa_0$-family, with shock
position $r_1$, then \eqref{eq:conservative-form} gives the exact identity
\be
\label{eq:mass-balance}
\int_a^b \frac{u_0(r)}{g(r)^2}\,dr
+ \int_a^b \frac{\delta(r)}{g(r)^2}\,dr
= \int_a^b \frac{u_\infty(r)}{g(r)^2}\,dr,
\ee
where $u_\infty$ denotes the limiting steady shock profile. Since $u_0$
and $u_\infty$ agree away from the interval between $r_0$ and $r_1$, and
have opposite signs on that interval, \eqref{eq:mass-balance} reduces to a
single scalar equation for $r_1$.


\subsection{Displacement law}
\label{section=7--2}

\subsubsection{Implicit formula}

\begin{claim}[Mass-conservation displacement law]
\label{thm:displacement}
Let $u_0$ be the steady shock~\eqref{eq:steady-shock-ic} and let
$\delta$ be a compactly supported perturbation with
$\Delta m := \int \delta(r)/g(r)^2\,dr$. If the perturbed solution
converges as $t\to\infty$ to a steady shock in the same $\kappa_0$-family
with shock position $r_1$, then $r_1$ satisfies the exact implicit
equation; more precisely, one has
\be
\label{eq:star}
\int_{r_0}^{r_1} \frac{2\,U_{\kappa_0}(r)}{g(r)^2}\,dr = \Delta m.
\ee
For small $|\Delta m|$, the displacement is linear to leading order:
\be
\label{eq:linear}
r_1-r_0 \;\approx\; \frac{g(r_0)^2}{2\,U_{\kappa_0}(r_0)}\;\Delta m.
\ee
\end{claim}

\begin{proof}
\bse
Immediate from \eqref{eq:mass-balance} and the sign structure of
the two steady shocks described above; \eqref{eq:linear} follows by
linearizing the left-hand side of \eqref{eq:star} in $r_1-r_0$.
\ese
\end{proof}

\subsubsection{Width of the linear regime}

\begin{claim}[Dependence of the linear regime on $\kappa_0$]
\label{prop:displacement-nonlinearity}
Let \(h(r)=2U_{\kappa_0}(r)/g(r)^2\) denote the integrand in
\eqref{eq:star}. For the relative variation of $h$ at $r_0$, one has
\be
\label{eq:displacement-nonlinearity-rate}
\frac{h'(r_0)}{h(r_0)}
= -\frac{g'(r_0)}{2\,U_{\kappa_0}(r_0)}
\left[
\frac{\kappa_0}{U_{\kappa_0}(r_0)}
+\frac{4\,U_{\kappa_0}(r_0)}{g(r_0)}
\right].
\ee
Consequently the characteristic displacement window on which the linear approximation
\eqref{eq:linear} remains accurate is of order
\be
\label{eq:linear-window-scale}
|r_1-r_0| \lesssim \frac{h(r_0)}{|h'(r_0)|}.
\ee
Equivalently, the corresponding weighted-mass window is of order
\(\displaystyle |\Delta m|\lesssim h(r_0)^2/|h'(r_0)|\).  For fixed $M$ and $r_0$, this window decreases as $\kappa_0$ increases.
\end{claim}

\begin{proof}
\bse
Differentiating $U_{\kappa_0}(r)=\sqrt{1-\kappa_0 g(r)}$ gives
$U_{\kappa_0}'(r)=-\kappa_0 g'(r)/(2U_{\kappa_0}(r))$.
Therefore, we obtain 
\be
\label{eq:proof-displacement-nonlinearity-derivative}
\frac{h'(r)}{h(r)}
= \frac{U_{\kappa_0}'(r)}{U_{\kappa_0}(r)} - \frac{2g'(r)}{g(r)}
= -\frac{\kappa_0 g'(r)}{2U_{\kappa_0}(r)^2}
  -\frac{2g'(r)}{g(r)},
\ee
which is equivalent to \eqref{eq:displacement-nonlinearity-rate}. Expanding
\eqref{eq:star} about $r_1=r_0$ gives
\be
\label{eq:proof-displacement-expansion}
\Delta m=h(r_0)(r_1-r_0)
+\frac12 h'(r_0)(r_1-r_0)^2+O((r_1-r_0)^3),
\ee
so the relative size of the first nonlinear correction is controlled by
$|h'(r_0)|/h(r_0)$. Finally,
$U_{\kappa_0}(r_0)^2=1-\kappa_0 g(r_0)$ decreases with $\kappa_0$, and
$\kappa_0/U_{\kappa_0}(r_0)^2$ is increasing on the admissible range. Hence
$|h'(r_0)/h(r_0)|$ increases with $\kappa_0$, which shrinks the linear
window.
\ese
\end{proof}
 

\subsection{Numerical experiments}
\label{section=7--3}

\subsubsection{Sampling protocol}

Each training sample is drawn by sampling
$M\sim\mathcal{U}(0.8,1.2)$,
$\kappa_0=K_0^2\sim\mathcal{U}(0.2,0.8)$, and
$r_0\sim\mathcal{U}(2.4,5.0)$, subject to $r_0>2M+0.3$. The steady shock
\eqref{eq:steady-shock-ic} is represented on
$r\in[2M+0.001,8.0]$ with $N=512$ points. The perturbation is a Gaussian
bump
\be
\label{eq:gaussian-perturbation}
\delta(r)=A\exp\!\left(-\frac12\left(\frac{r-r_c}{\sigma}\right)^2\right),
\ee
where $A\sim\mathcal{U}(-0.7,0.7)$ and
$\sigma\sim\mathcal{U}(0.05,0.4)$. The center $r_c$ is drawn, with equal
probability, from either
$[2M+3\sigma,r_0-3\sigma]$ or $[r_0+3\sigma,8.0-3\sigma]$. The perturbed
profile is clipped to $[-1,1]$, and samples whose numerical evolution does
not converge by $T_{\max}=25$ are discarded. The resulting validation set
contains $100$ samples, with $\Delta m\in[-22.2,7.4]$.

\subsubsection{Residual diagnostic}

To compare different candidates for the limiting location $r_1$, we
substitute each candidate into the exact identity~\eqref{eq:star}. This
gives the relative equation residual
\be
\label{eq:displacement-residual}
\varepsilon(r_1) := \frac{|H(r_1)-\Delta m|}{|\Delta m|},
\qquad
H(r_1)=\int_{r_0}^{r_1}\frac{2U_{\kappa_0}(r)}{g(r)^2}\,dr.
\ee
We evaluate this residual for three candidates: the linear approximation
\eqref{eq:linear}, the PITT prediction, and the finite-volume long-time
simulation $r_1^{\mathrm{sim}}$ reported in
\cite{lefloch2021classwellbalancedalgorithmsrelativistic}. The results are
summarized in Table~\ref{tab:equation-residual}, after binning the samples
according to $|\Delta m|$.

\begin{table}[ht]
\centering
\caption{Relative equation residual $\varepsilon(r_1)$ for three
candidates, stratified by $|\Delta m|$.}
\begin{tabular}{lccc}
\hline
$|\Delta m|$ bin& Linear & PITT & Numerical sim \\ \hline
$<1$  & $1.55\%$ & $7.56\%$  & $12.83\%$ \\
$[1,5)$ & $6.46\%$ & $3.69\%$  & $4.70\%$ \\
$\ge5$   & diverges & $18.32\%$ & $60.34\%$ \\ \hline
\end{tabular}
\label{tab:equation-residual}
\end{table}

\subsubsection{Interpretation of the residuals}

Three observations follow. First, the linear approximation's residual
diverges for $|\Delta m|\ge5$. Second, and most notably, the
\emph{simulation's own} residual reaches $60.34\%$ for $|\Delta m|\ge5$
-- substantially larger than at small $|\Delta m|$. Because this column
uses only $r_1^{\mathrm{sim}}$ (no model prediction at all), it is not
measuring prediction error; it is measuring how far the actual
long-time simulation outcome deviates from the same-$\kappa_0$-branch
assumption underlying Claim~\ref{thm:displacement}. A residual this
large is evidence that, for large perturbations, the perturbed solution
may no longer be relaxing to a steady shock within the original
$\kappa_0$-family at all, consistent with the regime-switching phenomenon
flagged in Remark~\ref{rem:branch-switching}. Third, PITT's residual is
smaller than the numerical simulation's residual in every bin
($7.56\%$ vs.\ $12.83\%$, $3.69\%$ vs.\ $4.70\%$, and $18.32\%$ vs.\
$60.34\%$).

\begin{remark}[Domain validity and branch switching]
\label{rem:branch-switching}
When $\Delta m$ is large enough that \eqref{eq:star} has no solution
within the admissible domain $D_{\kappa_0}$, the
assumption that the system relaxes to a shock within the same
$\kappa_0$-family breaks down; the perturbed solution then transitions to a
qualitatively different asymptotic branch, consistent with the
regime-change phenomena reported in~\cite{lefloch2021classwellbalancedalgorithmsrelativistic}.
We flag such cases explicitly (Figure~\ref{fig:displacement},
crimson markers) rather than extrapolating \eqref{eq:star} outside its
range of validity.
\end{remark}
\begin{figure}[t]
\centering
\includegraphics[width=0.95\textwidth]{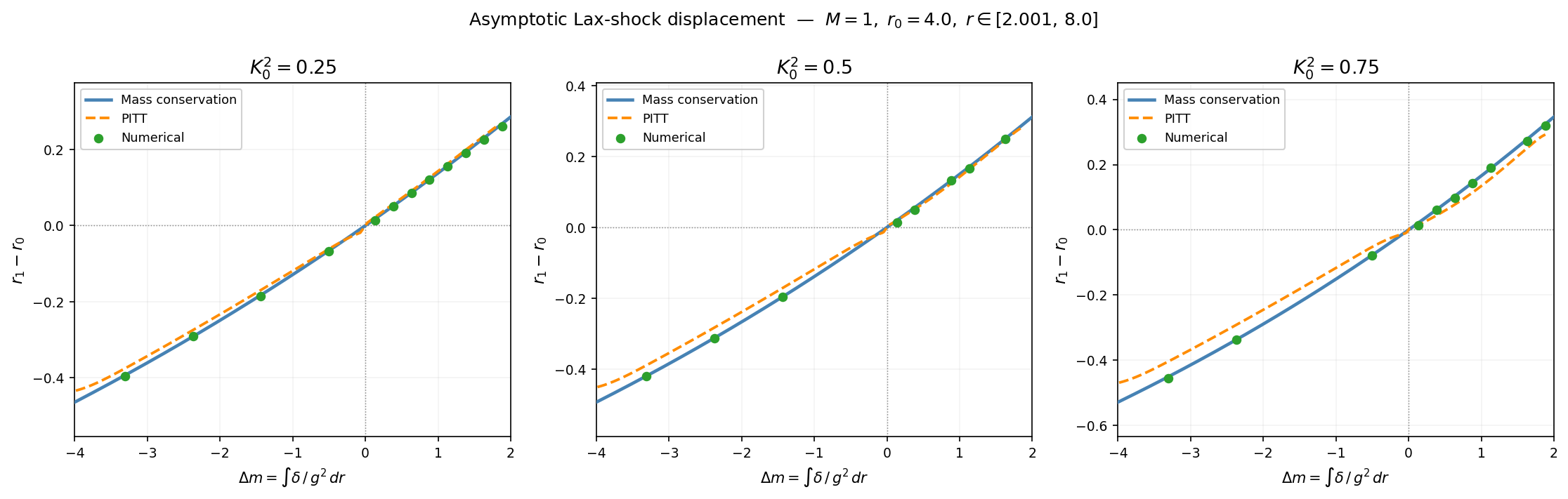}
\caption{Asymptotic shock displacement $r_1-r_0$ vs.\ weighted
perturbation mass $\Delta m=\int\delta(r)/g(r)^2\,dr$, at
$M=1,\,r_0=4.0$, for $K_0^2=0.25$ (in-distribution), $0.5$, and $0.75$.
Each panel shows $\Delta m\in[-4,2]$, well within the regime where the
mass-conservation formula~\eqref{eq:star} admits a solution.}
\label{fig:displacement}
\end{figure}

\subsubsection{Effect of the steady-state parameter}

Table~\ref{tab:nonlinearity} evaluates
\eqref{eq:displacement-nonlinearity-rate} for the three values of
$K_0^2=\kappa_0$ used in Figure~\ref{fig:displacement}. The characteristic
window $1/|h'(r_0)/h(r_0)|$ over which \eqref{eq:linear} remains accurate
shrinks monotonically as $\kappa_0$ increases, consistent with the
increasing curvature of the displacement curve across the three panels.
The mechanism is that larger $\kappa_0$ pushes
$U_{\kappa_0}(r_0)$ closer to zero, so the same weighted mass perturbation
produces a larger nonlinear response. This is the same sensitivity that
shrinks the outward threshold $\Delta m_{\max}$ in the last column of
Table~\ref{tab:nonlinearity}; both effects reflect proximity to the
boundary of $D_{\kappa_0}$ and are quantified by
Proposition~\ref{prop:displacement-nonlinearity}.

\begin{table}[h]
\centering
\caption{Nonlinearity rate~\eqref{eq:displacement-nonlinearity-rate}, characteristic
linear-regime window, and outward displacement threshold
$\Delta m_{\max}$ (Remark~\ref{rem:branch-switching}), all at $M=1,\,r_0=4.0$, for the three
values of $K_0^2$ shown in Figure~\ref{fig:displacement}. The linear
slope column matches the values already used to construct
Figure~\ref{fig:displacement}, confirming consistency with the earlier
computation.}
\begin{tabular}{cccccc}
\hline
$K_0^2$ & $U_{\kappa_0}(r_0)$ & Linear slope~\eqref{eq:linear} & $|h'(r_0)/h(r_0)|$ & Linear window $1/|h'/h|$ & $\Delta m_{\max}$ \\ \hline
0.25 & 0.935 & 0.1336 & 0.518 & 1.93 & 17.84 \\
0.50 & 0.866 & 0.1443 & 0.542 & 1.85 & 16.06 \\
0.75 & 0.791 & 0.1581 & 0.575 & 1.74 & 14.04 \\ \hline
\end{tabular}
\label{tab:nonlinearity}
\end{table}


\section{Concluding remarks}
\label{section=8}

We have developed a shock-aware PITT method for the
Schwarzschild--Burgers equation, using the explicit steady-state structure
and Rankine--Hugoniot dynamics of the model inside the neural rollout. The
numerical tests indicate that this analytical prior has the \emph{dominant role in
accuracy} near moving shocks, while equation tokenization and learned
corrections account for \emph{geometric and finite-resolution effects.}

The scalar nonlinear hyperbolic model studied here is simple from the viewpoint of hyperbolic
theory, but it does contain the features needed for the present purpose: a
non-trivial geometry, steady states, moving and stationary discontinuities,
rarefaction waves, and long-time relaxation. It therefore gives a useful
test case for PITT methods applied to fluid models with complex geometry or
complex dynamics.

The present results should also be read with their natural limitations.
The training data use the exact generalized Riemann structure of a scalar
balance law, and the numerical validation is mainly focused on
single-front dynamics and pre-interaction data. While the finite-volume
baseline used here already exhibits well-balanced behavior to machine
precision away from the horizon (\autoref{section=4--3}), a complete
assessment of the method will require comparison with solvers validated
specifically near the horizon, post-interaction tests, and repeated
runs with several random seeds. These points are not obstacles to the
construction proposed here; rather, they identify the next benchmarks
needed before the method can be regarded as a general-purpose solver for
hyperbolic balance laws.

The companion article~\cite{LeFlochXiang-2026b} will treat the
corresponding Schwarzschild--Euler fluid flows, and pursue the development of our PITT method. That problem requires handling a genuine system, several wave families, and interactions between
geometry, characteristic structure, and shock dynamics.

\paragraph*{Acknowledgments}

PLF was supported by the project ANR-23-CE40-0010-02 funded by the Agence Nationale de la Recherche, and the MSCA Staff Exchange Project 101131233 funded by the European Research Council. 



\appendix

\section{Machine learning terminology}
\label{section=A}

This appendix recalls a few machine learning-related terms used in the paper.

\paragraph{PINN.}
\emph{Physics-informed neural network}s were introduced as a way to train neural
networks while penalizing the \emph{residual} of a differential equation. They are
now a standard tool in scientific machine learning, especially for 
PDE problems and inverse problems, when sufficient smoothness is available. For \emph{discontinuous} solutions to nonlinear hyperbolic problems,
plain residual minimization is often insufficient, which is why the present
article uses \emph{analytical shock information} (based on steady states) inside the architecture.

\paragraph{Neural operator.}
A neural operator is a learning architecture designed to approximate maps
between function spaces, rather than maps between finite-dimensional
vectors alone. This viewpoint became prominent in the last decade and is
well adapted to parametrized PDEs, where the input may be an initial
condition, a coefficient, or a forcing term, and the output is a solution
field. In this article, the PITT methodology follows these lines but with \emph{additional
structure} (see next) from hyperbolic balance laws.

\paragraph{FNO.}
The Fourier neural operator is a neural-operator architecture based on
\emph{global Fourier modes}. It became a standard reference architecture for
learning solution operators of PDEs on grids. FNO layers are useful because
they encode \emph{nonlocal spatial information} efficiently, although Fourier
representations alone can be \emph{delicate near shocks}. Here the FNO component
encodes the sampled field and the geometry, while \emph{shock motion is handled
separately.} 

\paragraph{Transformer.}
A Transformer is a neural architecture built around attention mechanisms.
It was first developed in natural language processing and later became
standard across many domains involving structured sequences. In PDE
learning, the sequence may encode spatial points, equation tokens, or
latent variables. In this paper, the Transformer processes symbolic
equation tokens and supplies conditioning information to the rollout.

\paragraph{Token and equation tokenization.}
A token is a discrete symbol used as an elementary input to a Transformer.
Equation tokenization means representing the mathematical expression of a
PDE, its variables, coefficients, parameters, and structural relations as a
sequence of such symbols. This idea is recent in scientific machine
learning and aims to let the network distinguish equations by their
structure. Here tokenization records the balance law, the geometry, and the
shock relation.

\paragraph{Embedding and latent representation.}
An embedding is a learned vector representation of an object such as a
token, a parameter, or a sampled field value. A latent representation is an
internal set of variables used by the neural network, not necessarily meant
to coincide with a physical variable. These notions are standard in modern
machine learning. In this article, they are used to store information about
the equation, the field, and the Schwarzschild geometry during the rollout.

\paragraph{Self-attention and cross-attention.}
Self-attention lets elements of the same sequence exchange information;
cross-attention lets one representation query another representation. These
mechanisms are now standard in Transformer architectures. In the present
method, token summaries carry equation-level information, while the spatial
latent field carries solution-level information. Cross-attention transfers
the former into the latter at each update stage.

\paragraph{PITT.}
PITT stands for physics-informed token transformer. The idea is recent and
combines equation tokenization with Transformer-style processing of the
symbolic representation of a PDE. In the present article, PITT is adapted
to a setting with \emph{shock waves} by adding a \emph{Rankine--Hugoniot prior} and
\emph{steady-state profiles}. Thus the method is not only data-driven; it is also
guided by the analytical structure of the balance law.

\paragraph{Encoder and decoder.}
An encoder maps raw input data into latent variables, while a decoder maps
latent variables back to an output quantity of interest. This terminology
is standard in machine learning and appears in many architectures beyond
PDE learning. In the present setting, the field encoder processes the
initial profile and geometry, while the reconstruction stage plays the role
of a decoder for the predicted solution field.

\paragraph{Residual correction.}
A residual correction is a learned term added to a prescribed analytical
or numerical structure. This idea is common in hybrid scientific machine
learning, where a reliable model supplies the leading behavior and the
network learns what remains. In this paper, the Rankine--Hugoniot and
steady-state priors determine the dominant wave structure, while bounded
neural corrections account for discretization and unresolved effects.

\paragraph{Ablation study.}
An ablation study removes or disables selected components of a model in
order to measure their individual effect. This practice is standard in
machine learning, where architectures often contain several interacting
modules. In this article, ablations are used to compare the full PITT model
with variants without the Rankine--Hugoniot prior or without equation
tokenization. The results quantify which pieces of physics are most useful.

\paragraph{Training, validation, and test sets.}
These three datasets serve different roles in the learning protocol. The
training set is used to optimize the network parameters, the validation set
is used to tune choices and monitor generalization during training, and the
test set is reserved for final evaluation. This terminology is standard in
machine learning. In the present article, the split is chosen to test
generalization across physical parameters and solution regimes.

\paragraph{Out-of-distribution generalization.}
Out-of-distribution generalization refers to the behavior of a learned
model on parameters, geometries, or data regimes not fully represented
during training. The notion is central in modern machine learning because
good interpolation on a training distribution does not automatically imply
reliable extrapolation. For PDE applications, it is especially relevant
when the model is tested at new physical parameters or new wave
configurations. The present experiments include such tests for the
Schwarzschild--Burgers setting.

\end{document}